\documentclass[12pt,reqno]{amsart}

\usepackage{amsmath, amsfonts, amssymb, amsthm, mathrsfs}

\usepackage[utf8,utf8x]{inputenc}
\usepackage[pagebackref=true]{hyperref}
\hypersetup{%
	pdftitle = {},
	pdfkeywords = {},
	pdfsubject = {},
	pdfauthor = {},
	colorlinks = true,
	linktoc = page,
	citecolor = purple,
	linkcolor = blue 
}

\usepackage{xcolor}
\usepackage{caption}
\usepackage{tikz}
\usetikzlibrary{calc}
\usetikzlibrary{intersections}
\usetikzlibrary{cd}

\tikzcdset{scale cd/.style={every label/.append style={scale=#1},
		cells={nodes={scale=#1}}}}

\usepackage{enumerate}
\usepackage{mfirstuc}

\usepackage{todonotes}

\newcommand{\CC }{\mathbb{C}}
\newcommand{\RR }{\mathbb{R}}

\newcommand{\ZZ }{\mathbb{Z}}

\newcommand{\Ac }{\mathcal{A}}

\newcommand{\Ec }{\mathcal{E}}

\newcommand{\Vc }{\mathcal{V}}

\newcommand{\Lc }{\mathcal{L}}
\newcommand{\Nc }{\mathcal{N}}
\newcommand{\OM }{\mathcal{M}}

\newcommand{\Qc }{\mathcal{Q}}

\newcommand{\Tc }{\mathcal{T}}
\newcommand{\Sc }{\mathcal{S}}
\newcommand{\Hc }{\mathcal{H}}

\newcommand{\Hs }{\mathfrak{H}}

\newcommand{\isom }{\simeq}

\DeclareMathOperator{\sgn}{sgn}
\DeclareMathOperator{\mtc}{\underline{M}}

\DeclareMathOperator{\Zero}{\mathbf{0}}

\newcommand\wt[1]{\widetilde{#1}}
\newcommand\wh[1]{\widehat{#1}}

\DeclareMathOperator{\rk}{rk}

\DeclareMathOperator{\leT}{\dashv}

\DeclareMathOperator{\conv}{conv}
\DeclareMathOperator{\dist}{dist}

\DeclareMathOperator{\Pos}{\mathtt{Pos}}

\DeclareMathOperator{\Top}{\mathtt{Top}}

\def\dual{\vee}
\DeclareMathOperator{\aff}{aff}

\numberwithin{equation}{section}

\theoremstyle{plain}
\newtheorem{lemma}[equation]{Lemma}
\newtheorem{theorem}[equation]{Theorem}

\newtheorem{corollary}[equation]{Corollary}
\newtheorem{proposition}[equation]{Proposition}
\theoremstyle{definition}
\newtheorem{definition}[equation]{Definition}
\newtheorem{remark}[equation]{Remark}

\newtheorem{example}[equation]{Example}

\newtheorem{question}[equation]{Question}

\title[Modular flats and poset quasi-fibrations]
{Modular flats of oriented matroids and poset quasi-fibrations}

\author{Paul M\"ucksch}
\address{Paul M\"ucksch,
	Kyushu University,
	Department of Mathematics,
	744, Motooka, Nishi-ku,
	Fukuoka 819-0395, Japan}
\email{paul.muecksch+uni@gmail.com}

\begin{document}
	
\begin{abstract}
	We study the combinatorics of modular flats of oriented matroids and
	the topological consequences for their Salvetti complexes.
	We show that the natural map to the localized Salvetti complex 
	at a modular flat of corank one is what we call a poset quasi-fibration
	--  a notion derived from Quillen's fundamental Theorem B from algebraic $K$-theory.
	As a direct consequence, the Salvetti complex of an oriented matroid
	whose geometric lattice is supersolvable is a $K(\pi,1)$-space
	-- a generalization of the classical result for supersolvable hyperplane arrangements
	due to Falk, Randell and Terao. 
	Furthermore, the fundamental group of the Salvetti complex of a supersolvable oriented matroid
	is an iterated semidirect product of finitely generated free groups -- analogous to the realizable case.
	
	Our main tools are discrete Morse theory,
	the shellability of certain subcomplexes of the covector complex of an
	oriented matroid, a nice combinatorial decomposition of poset fibers of the localization map,
	and an isomorphism of covector posets associated
	to modular elements.
	
	We provide a simple construction of supersolvable oriented matroids.
	This gives many non-realizable supersolvable oriented matroids
	and by our main result aspherical CW-complexes.
\end{abstract}



\maketitle


\section{Introduction}
\label{sec:Introduction}

Complements of complex hyperplane arrangements provide a rich source of interesting topological spaces.
One of the most intricate and fascinating problems is to decide when arrangement complements are 
Eilenberg-MacLane $K(\pi,1)$-spaces where $\pi$ is the fundamental group of the space. 
By definition, this is the case if their universal covering space is contractible.
Synonyms are aspherical space or classifying space for the group $\pi$. 
They play an important role in algebraic topology e.g.\ as their cohomology coincides with
the group cohomology of $\pi$.
Moreover, if a $K(\pi,1)$-space is homotopy equivalent to a finite CW-complex then the group $\pi$ is torsion free.
In the following, we call complex arrangements whose complements are aspherical simply $K(\pi,1)$ or aspherical arrangements.

One of the most prominent positive results regarding $K(\pi,1)$-arran\-gements
is Deligne's theorem \cite{Deligne1973_SimplKpi1}, stating that complexified real simplicial arrangements
are aspherical.
Subsequently, Salvetti \cite{Salvetti1987_SalCpx} introduced a regular cell complex modeling the homotopy type of the
complement of a complexified real arrangement. 
Gel'fand and Rybnikov \cite{GelfandRybnikov1990_AlgTopInvOfOMs} realized that 
the construction of the Salvetti complex only depends on the oriented matroid
associated to the real arrangement (see Section~\ref{ssec:SalvettiCpx}).
This in turn led to a generalization of Deligne's result to oriented matroids in general, i.e.
the Salvetti complex of an oriented matroid whose covector complex is a simplicial cell decomposition of the sphere is indeed aspherical,
independently established by Cordovil \cite{Cordovil1994_SimplOMsKpi1} and Salvetti \cite{Salvetti1993_SimplOMsKpi1}.

The combinatorial notions of a modular flat in a geometric lattice and supersolvability 
were introduced by Stanley \cite{Stanley1971_ModularElts,Stan72_SSLat} (see Section~\ref{ssec:ModularSS}).
Regarding the $K(\pi,1)$ problem for hyperplane arrangements,
a further corner stone in the theory is due to Falk and Randell \cite{FalkRandell1985_FiberTypeArrs}
and Terao \cite{Terao1986_ModEltsFibrations}, see also \cite[Ch.~5]{OrTer92_Arr}. Thanks to their theorem, 
complex arrangements with supersolvable geometric lattices are aspherical.
Their proof is entirely geometric and does not use any complexes modeling the homotopy type of the complement.
More precisely, they showed that the natural projection of the ambient space of the arrangement
to the quotient space by a modular intersection of corank one, restricted to the complement space 
is a fiber bundle map with fiber a punctured complex plane. 
The associated long exact sequence of homotopy groups then yields their result.
As a special case for real arrangements one obtains that supersolvability of their intersection lattices
implies that their Salvetti complexes are concrete instances of finite $K(\pi,1)$-complexes.
A nice generalization of Falk, Randell and Terao's result to arrangements on connected abelian Lie groups 
is obtained in the very recent work of Bibby and Delucchi \cite{BibbyDelucchi2022_SSPosets}. 

The natural question appears whether the ``real'' version of Falk, Randell and Terao's theorem
extends to all oriented matroids as does Deligne's theorem.

It is exactly the aim of this paper to provide a positive answer to this question and a novel combinatorial
proof of the ``supersolvability theorem'' for real arrangements.
More precisely, we show that the natural projection map from the Salvetti complex to its localization
at a modular corank one flat in the geometric lattice of an oriented matroid (see Section~\ref{ssec:SalvettiCpx})
is what we call a \emph{poset quasi-fibration}. 
The poset fibers of this localization map are homotopy equivalent to the affine Salvetti complex
of an oriented matroid of rank two (Theorem~\ref{thm:MainQuasifibration}).
The notion of a poset quasi-fibration is directly derived from Quillen's fundamental
Theorem B from algebraic $K$-theory \cite{Quillen1973_KTheory1} (see Section~\ref{sec:QullenThmB}).

Since every poset quasi-fibration comes with a long exact homotopy sequence,
as an application we obtain the following theorem which indeed extends
Falk, Randell and Terao's result to oriented matroids.
\begin{theorem}
	\label{thm:SSOMKpi1}
	The Salvetti complex of a supersolvable oriented matroid is aspherical.
\end{theorem}

We also provide a simple argument that the Salvetti complex of each
localization of an oriented matroid with aspherical Salvetti complex is 
aspherical as well (Lemma~\ref{lem:Kpi1LocOMS}).
In the case of hyperplane arrangements this useful necessary condition for asphericity
is due to Oka , cf.\ \cite[Lem.~1.1]{Paris1993_DeligneCpx}.
Furthermore, we show that the fundamental group of the Salvetti complex of a supersolvable
oriented matroid is an iterated semidirect product of finitely generated free groups (see  Section~\ref{sec:FundGrps}) 
which is completely analogous to the realizable case.

A main ingredient in our study is discrete Morse theory which was established by
Forman \cite{Forman1998_DiscrMorse} (see Section \ref{ssec:DiscMorseTheory}). We use its reformulation in terms of acyclic matchings
on face posets of cell complexes due to Chari \cite{Chari2000_DiscreteMorseCombDecomp},
which is now standard for applications to combinatorial topology, see also Kozlov's book \cite{Kozlov2008_CombAlgTop}.
The acyclic matchings which are relevant to us are derived from the shellability
of the covector complex of an oriented matroid (see Section~\ref{ssec:Shellability}), 
a central result in the combinatorial topology of oriented matroids.

In the recent breakthrough work by Paolini and Salvetti \cite{PaoliniSalvetti2021_KpiAffArtin}
which settled the $K(\pi,1)$ conjecture for affine Artin groups, discrete Morse theory
in combination with certain shellability results also played an important role.

In a separate erratum \cite{DelucchiMuecksch2022_Erraturm-ShellableSalvetti} 
we comment on results stated by Delucchi in \cite{Delucchi2008_ShellableSalvetti} which are unfortunately not correct.
Resolving these issues would lead to a generalization of our results, similar to the more general
fiber bundles for hyperplane arrangements described in the work of Falk and Proudfoot \cite{FalkProudfoot2002_BundlesOfArrs}.

\bigskip
This article is organized as follows.
In Section~\ref{sec:Preliminaries} we recall all the required notions from combinatorial
topology and oriented matroid theory.
In Section~\ref{sec:CombModularOMS} we further elaborate on the special properties
of modular flats of oriented matroids.
Section~\ref{sec:StratSalvettiFibers} describes a useful stratification of the poset fibers of the
localization map between Salvetti complexes.
Section~\ref{sec:QullenThmB} introduces the special class of order
preserving maps which play a central role in our study.
In Section~\ref{sec:MatchingsAndPosetQFibrations} we derive certain acyclic matchings which are then applied
to prove our main result: Theorem~\ref{thm:MainQuasifibration}.
In the following Section~\ref{sec:FundGrps} we describe the structure of the fundamental group of
the Salvetti complex of a supersolvable oriented matroid.
To conclude, in Section~\ref{sec:ApplicationExamples} we provide a simple construction
which yields many non-realizable supersolvable oriented matroids.


\section{Preliminaries}
\label{sec:Preliminaries}

In this preliminary section we review all the objects and notions required 
for our study.

\subsection{Posets and regular CW-complexes}
\label{ssec:PosetTopology}

We introduce some notation and collect basic informations on the combinatorial topology of posets and regular CW-complexes.
Throughout this paper, all partially ordered sets (or posets for short) and all complexes are assumed to be finite.
A concise reference for us is \cite[App.~4.7]{BLSWZ1999_OrientedMatroids}.

We start by recalling some basic poset terminology.
Let $P = (P,\leq)$ be a poset. Its \emph{dual} or \emph{opposite} poset $P^\dual$
has the same ground set as $P$ but its order relation $\leq^\dual$ 
is reversed: $x \leq^\dual y :\iff y \leq x$.

To emphasize the order relation of a particular poset $P$ we occasionally use the notation $\leq_P$.

We denote by $\lessdot$ the \emph{cover relations} of $P$.
That is $x \lessdot y$ if and only if $x < y$, and for all $z \in P$:
$x < z \leq y$ implies $z=y$.

A subset $I \subseteq P$ is called an \emph{order ideal} if $y \in I$ and $x \in P$ with
$x \leq y$ implies $x \in I$.
Dually, $F \subseteq P$ is called an \emph{order filter} if $F^\dual$ is an order ideal in $P^\dual$.
We write $P_{\leq y} = \{x \in P \mid x \leq y\}$ for the \emph{principal} order ideal generated by $y \in P$
and similarly $P_{\geq x} = \{ y \in P \mid x \leq y\}$ for the principal order filter
generated by $x \in P$.

A map $f:P \to Q$ between two posets is \emph{order preserving} or a \emph{poset map}
if for all $x \leq_P y$ $(x,y \in P)$ we have $f(x) \leq_Q f(y)$.
If $q \in Q$ then we write $(f\downarrow q) := f^{-1}(Q_{\leq q})$ for the \emph{poset fibers} of $f$.

A \emph{linear extension} $\dashv$ of $P$ is a total order on the ground set $P$ such that 
for all $x,y \in P$: $x \leq y$ implies $x \dashv y$. Linear extensions of finite posets always exits.
Given an order ideal $I \subseteq P$, a linear extension $\dashv_I$ on $I$ 
can always be extended to a linear extension $\dashv$ on $P$ with
$x \dashv y$ if and only if $x \dashv_I y$ for all $x,y \in I$.
Moreover, a linear extension on $I$ can always be extended to $P$ in such a way
that all the elements of $I$ come first. 
An analogue statement holds for order filters.

\bigskip

We assume familiarity with the notion of \emph{CW-complexes}, also called \emph{cell complexes}, e.g.\ see \cite{Hatcher2002_AT}.
Important for our context are the \emph{regular} cell complexes
since they are completely determined by their \emph{face poset} of closed cells ordered by inclusion.
Let $\Sigma$ be a regular cell complex and $\sigma,\tau \in \Sigma$.
If $\tau \subseteq \overline{\sigma}$ then $\tau$ is called a \emph{face}
of $\sigma$ and we simply write $\tau \leq \sigma$.
In what follows, we thus regard a regular cell complex $\Sigma$ itself as a poset.
Cells of dimension zero are referred to as \emph{vertices}.

Conversely, to every poset $P$ is associated its \emph{order complex}
$\Delta(P)$. 
It is the simplicial complex with 
$n$-simplices given by all chains of length $n$ in $P$, i.e.
\[
	\Delta(P)_n := \{\{x_0,\ldots,x_n\} \mid x_i \in P\text{ with }x_0 < x_1 < \ldots < x_n\}.
\]

For a simplicial complex $\Delta$ we denote its topological realization by $|\Delta|$.
The composition of forming the order complex and taking its realization yields the 
\emph{simplicial realization} functor 
\[
	|\Delta(-)|:\Pos \to \Top
\]
from the category $\Pos$ of posets and order preserving maps
to the category $\Top$ of topological spaces and continuous maps. 
For a regular cell complex $\Sigma$ the image of its face poset under this functor
is naturally homeomorphic to $\Sigma$: $|\Delta(\Sigma)| \cong \Sigma$.

We call an order preserving map $f:P \to Q$ a homotopy equivalence
provided $|\Delta(f)|$ is a homotopy equivalence.

Maps between finite regular cell complexes are given as order preserving maps
between their face posets.	
A poset map between the face posets of
regular cell complexes which is a homotopy equivalence in the above sense
is indeed a homotopy equivalence of their defining topological spaces.

A regular cell complex $\Sigma$ is called \emph{pure} if
all maximal cells have the same dimension.

A subset $\Gamma \subseteq \Sigma$ is called a (closed) subcomplex
if with a cell $\sigma \in \Gamma$ all of its faces also belong to $\Gamma$,
i.e.\ $\Gamma$ is an order ideal of $\Sigma$.
For a subset of cells $S \subseteq \Sigma$
we denote by $\Sigma(S) := \{\tau \in \Sigma \mid \tau \leq \sigma$ for a $\sigma \in S\}$ the subcomplex
of all faces of cells in $S$.
In poset terminology, $\Sigma(S) = \bigcup_{\sigma \in S} \Sigma_{\leq \sigma} $ is the order ideal generated by $S$.
For a cell $\sigma \in \Sigma$ we denote by $\delta\sigma$ the subcomplex consisting of
all proper faces of $\sigma$.


\subsection{Oriented Matroids}
\label{ssec:OrientedMatroids}

We recall some basics from oriented matroid theory.
An oriented matroid can be regarded as a combinatorial abstraction of
a real hyperplane arrangement, i.e.\ a finite set of hyperplanes in a finite dimensional real vector space.
The standard reference is the book by Bj\"orner et al.\ \cite{BLSWZ1999_OrientedMatroids}.
In our study, we adopt the point of view that oriented matroids constitute a
special class of regular cell complexes, see Remark~\ref{rem:OMTopRep} below.

First, let us recall some notation for sign vectors.
Let $E$ be some finite set and consider a sign vector $\sigma \in \{+,-,0\}^E$.
We denote by $z(\sigma) := \{ e \in E \mid \sigma_e = 0\}$ its \emph{zero set}.
The \emph{opposite} vector $-\sigma$ is defined by
\[
(-\sigma)_e := \begin{cases}
	- & \text{ if } \sigma_e = +, \\
	+ & \text{ if } \sigma_e = -, \\
	0 & \text{ if } \sigma_e = 0.
\end{cases}
\]
The \emph{composition} $\sigma \circ \tau$ of two sign vectors $\sigma, \tau \in \{+,-,0\}^E$  
is defined by
\[
(\sigma \circ \tau)_e := \begin{cases}
	\sigma_e &\text{ if } \sigma_e \neq 0, \\
	\tau_e & \text{ else},
\end{cases}
\]
and their \emph{separating set} $S(\sigma,\tau)$ is defined as
\[
S(\sigma,\tau) := \{e \in E \mid \sigma_e = -\tau_e \neq 0\}.
\]
If $A \subseteq E$ and $\sigma \in \{+,-,0\}^E$ we write $\sigma|_A := (\sigma_e \mid e \in A) \in \{+,-,0\}^A$
for its restriction to $A$.

\begin{definition}
\label{def:OMAxioms}
	An \emph{oriented matroid} $\OM = (E,\Lc)$ is given by a finite
	ground set $E$ and a subset of sign vectors $\Lc \subseteq \{+,-,0\}^E$ called \emph{covectors} of $\OM$,
	subject to satisfying the following axioms:
	\begin{enumerate}[(1)]
		\item $\mathbf{0} := (0,0,\ldots,0) \in \Lc$,
		\item if $\sigma \in \Lc$, then $-\sigma\in \Lc$,
		\item if $\sigma, \tau \in \Lc$, then $\sigma \circ \tau \in \Lc$,
		\item if $\sigma, \tau \in \Lc$ and $e \in S(\sigma,\tau)$
		there exists an $\eta \in \Lc$ such that $\eta_e = 0$ and
		$\eta_f = (\sigma \circ \tau)_f = (\tau \circ \sigma)_f$ for all $f \in E \setminus S(\sigma,\tau)$.
	\end{enumerate}
	
	Let $\leq$ be the partial order on the set $\{+,-,0\}$ defined by $0 < +$, $0 < -$, with
	$+$ and $-$ incomparable. This induces the product partial order on $\{+,-,0\}^E$ in which sign vectors
	are compared component wise.
	The \emph{face poset} or covector poset $(\Lc, \leq)$ of $\OM$ is obtained by considering the induced partial order on the
	subset $\Lc \subseteq \{+,-,0\}^E$.
	
	The \emph{rank} of $\OM$ is defined as the length of a maximal chain in $\Lc$.
	
	The maximal elements in $\Lc$ are called \emph{topes}; they are denoted by $\Tc = \Tc(\OM)$.
\end{definition}

Note that this is only one of several equivalent ways to define oriented matroids, cf.\ \cite[Ch.~3]{BLSWZ1999_OrientedMatroids}.

An element $e \in E$ is called a \emph{loop} of the oriented matroid $\OM =(E,\Lc)$
if $\sigma_e = 0$ for all $\sigma \in \Lc$.
Two elements $e,f \in E$ which are not loops are \emph{parallel} if $\sigma_e = 0$
implies $\sigma_f=0$ for all $\sigma \in \Lc$.

In the rest of the sequel, without loss of generality, we always assume an oriented matroid to
be simple, i.e.\ it contains no loops or parallel elements.
This is justified by the fact that the covector poset of an oriented matroid and its simplification are isomorphic, 
cf.\ \cite[Lem.~4.1.11]{BLSWZ1999_OrientedMatroids}.

We have the following principal examples of oriented matroids.
\begin{example}
	\label{ex:OMRealArrangement}
	Let $\Ac$ be an arrangement of hyperplanes in a real vector space $V \isom \RR^\ell$.
	For $H \in \Ac$ choose defining linear forms $\alpha_H \in V^*$ with $\ker(\alpha_H) = H$.
	Then $\OM(\Ac) = (\Ac,\Lc(\Ac))$ with 
	\[
		\Lc(\Ac) := \{ (\sgn(\alpha_H(v)) \mid H \in \Ac) \mid v \in V\} \subseteq \{+,-,0\}^\Ac
	\]
	is an oriented matroid.
\end{example}

We recall the definition of an affine oriented matroid.

\begin{definition}
	\label{def:AffOM}
	Let $\OM = (E,\Lc)$ be an oriented matroid and $g \in E$.
	Define $\OM_{\aff, g}:= (E,\Lc^+,g)$ by
	\[
		\Lc^+ := \{\sigma \in \Lc \mid \sigma_g = +\}.
	\]
	Then $\OM_{\aff,g}$ is called an \emph{affine} oriented matroid
	or the \emph{decone} of $\OM$ with respect to $g$.
\end{definition}

\begin{remark}
	\label{rem:OMTopRep}
	For our purposes, oriented matroids are best thought of as special instances of regular cell complexes.
	This is justified by the topological representation theorem due to Folkman and Lawrence \cite{FolkmanLawrence1978_OMs},
	see also \cite[Thm.~5.2.1]{BLSWZ1999_OrientedMatroids}.
	Its core implication that the reduced covector poset $\Lc\setminus\{\Zero\}$ 
	is the face poset of a regular cell decomposition of a sphere induced by
	an arrangement of tamely embedded codimension one subspheres suffices for our study, see also Theorem~\ref{thm:OMCovectorShellable}.
	Moreover, $\Lc\setminus\{\Zero\}$ is a PL-sphere, so $\Lc^\dual$ represents a regular cell decomposition
	of a ball, cf.\ \cite[Prop.~4.7.26]{BLSWZ1999_OrientedMatroids}.
\end{remark}

\bigskip

There is another poset associated to an oriented matroid $\OM = (E,\Lc)$ which is essential for our considerations.
\begin{definition}
	The \emph{geometric lattice} $L(\OM)$ of $\OM$ is defined as
	\[
	L(\OM) := \{ z(\sigma) \mid \sigma \in \Lc \} \subseteq 2^E,
	\]
	with partial order given by inclusion.
	Elements of $L(\OM)$ are called \emph{flats}.
	For two elements $X,Y \in L(\OM) $, their join is $X \vee Y = \inf\{Z \in L(\OM) \mid X\cup Y \subseteq Z\}$,
	their meet is $X \wedge Y = X \cap Y$.
	 
	Moreover, for $X,Y \in L=L(\OM)$ we write $L_X := L_{\leq X}$ for the lattice of the \emph{localization} at $X$,
	$L^Y := L_{\geq Y}$ for the lattice of the \emph{contraction} to $Y$ 
	and $L_X^Y := [X,Y] := L_{\leq X} \cap L_{\geq Y}$ for an interval.
\end{definition}

The lattice $L=L(\OM)$ is graded by \emph{rank}, i.e.\
the rank of $X \in L$ is the length of a maximal chain in $L_{\leq X}$.

Note that the map $z:\Lc \to L$ is a cover and rank preserving, order reversing surjection \cite[Prop.~4.1.13]{BLSWZ1999_OrientedMatroids}.
The associated order preserving map $\Lc^\vee \to L$ is also denoted by $z$.

\begin{definition}[Restriction, contraction, and localization]
	\label{def:DelitionContraction}
	Let $\OM = (E,\Lc)$ be an oriented matroid.
	\begin{enumerate}
		\item
		For a subset $A \subseteq E$ the \emph{restriction} $\OM|_A = (A,\Lc|_A)$ of $\OM$ to $A$ is
		defined by
		\[
			\Lc|_A := \{\sigma|_A \mid \sigma \in \Lc\}.
		\]
		
		\item
		For $X \subseteq E$ the \emph{contraction} $\OM/X = (E\setminus X, \Lc/X)$ is defined by
		\[
		\Lc/X := \{ \sigma|_{E\setminus X} \mid \sigma \in \Lc \text{ with } X \subseteq z(\sigma)\}.
		\]
		
		\item 
		For a flat $X \in L(\OM)$ the \emph{localization} of $\OM$ at $X$ is $\OM_X := \OM|_X$.
		The natural projection of $\Lc$ to $\Lc_X := \Lc|_X$ is denoted by $\rho_X:\Lc \to \Lc_X, \sigma \mapsto \sigma|_X$.
	\end{enumerate}
	
\end{definition}

\begin{remark}
	\label{rem:CovectorsContraction}
	For the covectors of the contraction $\OM/X$ we also use the notation
	$\Lc^X := \Lc/X$ where $L^X = \{Y \in L(\OM) \mid Y \geq X\}$.
	We often identify $\Lc^X$ with $z^{-1}(L^X)$ by the obvious isomorphism 
	$z^{-1}(L^X) \isom \Lc^X$.
	Therefor, we write $\Lc^X \subseteq \Lc$ or $\Lc^X \hookrightarrow \Lc$.
	
	If $L_X^Y = [Y,X] \subseteq L(\OM)$ we frequently write $\Lc_X^Y := \Lc_X/Y$.
\end{remark}

\begin{remark}
	\label{rem:SectionRhoX}
	Let $X \in L$ and let $\rho_X: \Lc \to \Lc_X$ be the localization map.
	Then for any $\alpha \in \Lc$ with $z(\alpha) = X$ we have a section
	$\iota_\alpha:\Lc_X \to \Lc$ to $\rho_X$ which is defined by:
	\[
	\iota_\alpha(\sigma)_e := \begin{cases}
		\sigma_e &\text{if } e \in X,\\
		\alpha_e &\text{else}.
	\end{cases}
	\]
	An easy calculation shows that both the localization map $\rho_X$ and the section $\iota_\alpha$ preserve
	the composition of covectors, i.e.\ we have $\rho_X(\sigma\circ\tau) = \rho_X(\sigma)\circ\rho_X(\tau)$
	for all $\sigma, \tau \in \Lc$,
	and $\iota_\alpha(\sigma'\circ\tau') =   \iota_\alpha(\sigma')\circ\iota_\alpha(\tau')$
	for all $\sigma',\tau' \in \Lc_X$.
\end{remark}

We recall the notion of an extension of an oriented matroid.
\begin{definition}
	\label{def:OMExtension}
	Let $\OM = (E,\Lc)$ be an oriented matroid.
	Then another oriented matroid $\wt{\OM} = (E',\Lc')$ with $E \subseteq E'$ is called an \emph{extension} of $\OM$
	if $\wt{\OM}|_E = \OM$.
\end{definition}

If $\wt{\OM}$ is an extension of $\OM$, then we have a natural injective map
$L(\OM) \hookrightarrow L(\wt{\OM}), X \mapsto \wt{X}$,
where $\wt{X} = \inf\{X' \in L(\wt{\OM}) \mid X \subseteq X'\}$.


\subsection{Modular flats and supersolvable geometric lattices}
\label{ssec:ModularSS}

For this section, let $L = L(\OM)$ be the geometric lattice of a simple oriented matroid $\OM = (E,\Lc)$.
The notion of a modular flat in a geometric lattice was introduced
by Stanley \cite{Stanley1971_ModularElts}.

\begin{definition}
	\label{def:ModularElementGeometricLattice}
	Let $X \in L$.
	Then $X$ is \emph{modular} if
	for all $Y,Z \in L$ with $Z \leq Y$ we have
	\[
	Z \vee (X \wedge Y) = (Z\vee X)\wedge Y.
	\]
%
\end{definition}

The following notion is also due to Stanley \cite{Stan72_SSLat}.
\begin{definition}
	\label{def:SupersolvableL}
	The geometric lattice $L$ is called \emph{supersolvable} if there is a maximal chain
	$X_0 < X_1 < \ldots < X_r$ of flats in $L$ such that $X_i$ is modular for all $0\leq i \leq r$.
	
	In this case, we call $\OM$ itself supersolvable.
\end{definition}

Supersolvable oriented matroids of rank $3$ are easily describes by the following lemma,
the proof of which is left to the reader.
\begin{lemma}
	\label{lem:SSLrk3}
	Assume $\OM$ to be of rank $3$.
	Then $\OM$ is supersolvable if and only if there is a flat $X \in L$ of rank $2$
	such that for all flats $Y \in L$ of rank $2$ we have 
	$X \cap Y \neq \emptyset$.
\end{lemma}

Modular flats have the following property (which is actually equivalent to modularity).
\begin{lemma}[{\cite[Thm.~3.3]{Brylawski1975_ModularConstr}}]
	\label{lem:IntervalsModGeomL}
	Let $X \in L$ be a modular flat. 
	Then for $Y \in L$ the map $p_X:[Y,X\vee Y] \to [X \wedge Y, X], Z \mapsto Z \wedge X$ is a poset isomorphism
	with inverse $s_Y:[X \wedge Y, X] \to [Y,X\vee Y], W \mapsto W \vee Y$.
\end{lemma}


\subsection{Topes and convex subsets}

Recall that we denote the topes of an oriented matroid $\OM = (E,\Lc)$,
i.e.\ the maximal elements of $\Lc$ by $\Tc$.
As remarked earlier, without loss of generality, we assume that $\OM$ is simple.

\begin{definition}
	\label{def:TopePoset}
	Let $B \in \Tc$.
	We define a partial order on $\Tc$ by
	\[
	R \leq T :\iff S(B,R) \subseteq S(B,T).
	\]
	The resulting ranked poset $\Tc(\OM,B)$  with rank function $\rk(T) := |S(B,T)|$
	is called the \emph{tope poset} with respect to $B$.
\end{definition}

\begin{definition}
	\label{def:ConvexSubsets} 
	Let $\Tc$ be the topes of an oriented matroid $\OM = (E,\Lc)$.
	\begin{enumerate}[(i)]
		\item 
		Let $T,R \in \Tc$ be two topes. 
		Then the \emph{distance} $\dist(T,R)$ of $T$ and $R$ is
		defined as $\dist(T,R) := |S(T,R)|$.
		
		\item
		Let $e \in E$. 
		Denote by $\Tc_e^+ := \{ T \in \Tc \mid T_e = +\} \subseteq \Tc$ the \emph{positive halfspace}
		with respect to $e$.
		The \emph{negative halfspace} $\Tc_e^- \subseteq \Tc$ is defined analogously.
		Set $\Hs = \Hs(\OM) = \{T_e^\epsilon \mid e\in E, \epsilon \in \{+,-\} \}$,
		and an $\Hc \in \Hs$ is called a \emph{halfspace} of $\OM$.
		
		\item
		Let $\Qc \subseteq \Tc$ be a subset of topes.
		Then the \emph{convex hull} $\conv(\Qc)$ of $\Qc$ is defined as the intersection
		of all halfspaces containing $\Qc$:
		\[
		\conv(\Qc) = \bigcap\limits_{\Qc \subseteq \Hc, \Hc \in \Hs} \Hc.
		\]
		
		\item 
		A subset of topes $\Qc \subseteq \Tc$ is called \emph{convex} if 
		$\conv(\Qc) = \Qc$.
		
		\item 
		For a subset $\Qc \subseteq \Tc$ we set 
		$\Lc(\Qc) := \{\sigma \in \Lc \mid \sigma \leq T$ for some $T \in \Qc\}$
		and
		$\Lc^\dual[\Qc] := \{\sigma \in \Lc^\dual \mid \Tc(\sigma) \subseteq \Qc\}$,
		where $\Tc(\sigma) = \{T \in \Tc \mid \sigma \leq_{\Lc} T\}$.
		
	\end{enumerate}
\end{definition}

\begin{remark}
	\label{rem:ConvexSubsetsDualCpx}
	Note that the definition of $\Lc^\dual[\Qc] \subseteq \Lc^\dual$ yields
	a subcomplex since $\Lc^\dual$ has the \emph{intersection property}, i.e.\
	a face is uniquely determined by its vertices.
	Furthermore, an easy calculation shows that $(\Lc \setminus \Lc(\Tc\setminus\Qc))^\dual = \Lc^\dual[\Qc]$.
\end{remark}

We record the following alternative characterization of convex subsets.
\begin{lemma}{\cite[Prop.~4.2.6]{BLSWZ1999_OrientedMatroids}}
	\label{lem:ConvSubsetDist}
	Let $\Qc \subseteq \Tc$.
	Then $\Qc$ is convex if and only if 
	$T,R \in \Qc$ and $Q \in \Tc$ with $\dist(T,Q) + \dist(Q,R) = \dist(T,R)$
	implies $Q \in \Qc$.
\end{lemma}

Two easy but useful facts are provided by the following lemmas.
\begin{lemma}
	\label{lem:LinExtTConvexSubset}
	Let $\Qc \subseteq \Tc$ be a convex subset of topes.
	Then for a tope $B \in \Qc$ there is a linear extension $\leT$ of $\Tc(\OM,B)$ such 
	that with respect to $\leT$ all the topes in $\Qc$ come first.
\end{lemma}
\begin{proof}
	The convexity of $\Qc$ and Lemma \ref{lem:ConvSubsetDist} imply that $\Qc$ forms an order ideal in
	$\Tc(\OM,B)$ for every $B \in \Qc$, cf.\ \cite[pp.~172ff.]{BLSWZ1999_OrientedMatroids}. This, in turn, implies the statement.
\end{proof}

We consider the localization map $\rho_X:\Lc^\dual \to \Lc_X^\dual$.
For the subset of topes (vertices of $\Lc^\dual$) 
in the poset fiber $(\rho_X\downarrow \sigma)$ we write $\Tc(\rho_X\downarrow \sigma)$.

\begin{lemma}
	\label{lem:FiberPosetTopesConvex}
	Let $\sigma$ in $\Lc_X^\dual$.
	Then $\Tc(\rho_X\downarrow \sigma)$ is a convex subset of $\Tc$.
\end{lemma}
\begin{proof}
	We have $\Tc(\rho_X\downarrow\sigma) = \bigcap\limits_{e \in X\setminus z(\sigma)} \Tc_e^{\sigma_e}$.
\end{proof}


\subsection{Shellability and subcomplexes}
\label{ssec:Shellability}

We briefly elaborate on shellability of regular cell complexes, cf.\ \cite[App.~4.7]{BLSWZ1999_OrientedMatroids}.

\begin{definition}
	\label{def:ShellableCpx}
	Let $\Sigma$ be a pure $d$-dimensional regular cell complex.
	A linear ordering $\sigma_1,\sigma_2,\ldots,\sigma_t$ of its maximal cells is called
	a \emph{shelling} if either $d=0$, or $d\geq 1$ and the following conditions are satisfied:
	\begin{enumerate}[(i)]
		\item $\delta\sigma_j \cap (\bigcup_{i=1}^{j-1}\delta\sigma_i)$ is pure of dimension $d-1$ for $2\leq j \leq t$,
		
		\item $\delta\sigma_j$ has a shelling in which the $(d-1)$-cells of $\delta\sigma_j \cap (\bigcup_{i=1}^{j-1}\delta\sigma_i)$ 
		come first for $2 \leq j \leq t$,
		
		\item $\delta\sigma_1$ has a shelling.
	\end{enumerate} 
	A complex which admits a shelling is called \emph{shellable}.
\end{definition}

For our investigation, the principal example of a shellable regular cell complex is the
covector complex $\Lc$ of an oriented matroid.
More precisely, we have the following theorem.
\begin{theorem}[{\cite[Thm.~4.3.3]{BLSWZ1999_OrientedMatroids}}]
	\label{thm:OMCovectorShellable}
	The reduced covector poset $\Lc \setminus \{\Zero\}$ of an oriented matroid of rank $r$ is isomorphic to
	the face poset of a shellable regular cell decomposition of the $(r-1)$-sphere.
	Furthermore, every linear extension of $\Tc(\OM,B)$ is a shelling of $\Lc \setminus \{\Zero\}$.
\end{theorem}

We list the following useful fact, cf.\ \cite[Prop.~4.7.26]{BLSWZ1999_OrientedMatroids}.
\begin{proposition}
	\label{prop:ShellingSphereSubsequenceBall}
	Let $\Sigma$ be a regular cell decomposition of the $d$-sphere which has a shelling
	$\sigma_1,\ldots,\sigma_t$.
	Then for $k <t$ the subcomplex $\Sigma(\{\sigma_1,\ldots,\sigma_k\})$
	is a shellable $d$-ball with shelling $\sigma_1,\ldots,\sigma_k$.
\end{proposition}

With the preceding results we can delineate the following property of subcomplexes
generated by subsets of the topes $\Tc = \Tc(\OM)$ of an oriented matroid $\OM$. 
\begin{corollary}
	\label{coro:ShellableComplTConvex}
	Let $\emptyset \neq \Qc \subseteq \Tc$ be a convex subset of topes.
	Then both the subcomplexes $\Lc(\Qc)$ and $\Lc(\Tc\setminus\Qc)$ are shellable balls.
\end{corollary}
\begin{proof}
	Let $B \in \Qc$.
	By Lemma~\ref{lem:LinExtTConvexSubset} there is a linear extension of $\Tc(\OM,B)$ such that all the topes
	in $\Qc$ come first.
	Conversely, with respect to the opposite linear extension of $\Tc(\OM,-B)$
	all the topes in $\Tc\setminus\Qc$ come first. 
	This, in turn, yields the statement with Theorem~\ref{thm:OMCovectorShellable}
	and Proposition~\ref{prop:ShellingSphereSubsequenceBall}.  
\end{proof}


\subsection{The Salvetti complex of an oriented matroid}
\label{ssec:SalvettiCpx}

We start by defining a certain poset which was first constructed by
Salvetti \cite{Salvetti1987_SalCpx} for an oriented matroid $\OM(\Ac)$
associated to a real hyperplane arrangement $\Ac$.
Gel'fand and Rybnikov \cite{GelfandRybnikov1990_AlgTopInvOfOMs} observed,
that the same construction applies more generally to oriented matroids.

\begin{definition}
	\label{def:SalvettiCpxOM}
	Let $\Lc$ be the poset of covectors of an oriented matroid $\OM$.
	Then the \emph{Salvetti poset} $\Sc = \Sc(\OM)$ is defined
	as
	\[
		\Sc := \{ (\sigma,T) \mid T \in \Tc\text{ and } \sigma \in \Lc_{\leq T} \} \subseteq \Lc \times \Tc,
	\]
	with partial order
	\[
	(\sigma, T) \leq_\Sc (\tau, R) :\iff \sigma \geq_\Lc \tau \text{ and } \sigma \circ R = T. 
	\]
	
	The \emph{affine} Salvetti poset $\Sc_{\aff,g}$ of the affine oriented matroid $\OM_{\aff,g}$ is defined as
	$\Sc_{\aff,g} := \{(\sigma,T) \in \Sc \mid \sigma_g=+, T_g=+\}$.
\end{definition}

The following theorem due to Salvetti \cite{Salvetti1987_SalCpx} 
is a fundamental result in the homotopy theory of complements
of complexified real arrangements. 

\begin{theorem}[{\cite{Salvetti1987_SalCpx}}]
	The Salvetti poset is the face poset of a regular cell complex.
	If $\OM=\OM(\Ac)$ is the oriented matroid associated to a hyperplane arrangement $\Ac$
	in a real vector space $V$,
	then the Salvetti complex is homotopy equivalent to the complement of the complexified arrangement:
	\[
		\Sc(\OM) \cong V\otimes\CC \setminus \left( \bigcup\limits_{H \in \Ac} H\otimes\CC \right).
	\]
\end{theorem}

Let $X \in L(\OM)$. The projection map $\rho_X:\Lc \to \Lc_X$ 
respects the composition of covectors by Remark~\ref{rem:SectionRhoX}.
Hence, $\rho_X$ induces an order preserving map 
\[
	\wt{\rho}_X: \Sc = \Sc(\OM) \to \Sc(\OM_X) =: \Sc_X, (\sigma,T) \mapsto (\rho_X(\sigma),\rho_X(T))
\]
between the Salvetti complexes.
Moreover, for each $\alpha \in \Lc$ with $z(\alpha) = X$ we get a section 
$\iota_\alpha:\Lc_X \to \Lc$ to $\rho_X$ by Remark~\ref{rem:SectionRhoX} which respects the composition of covectors.
Hence, this also yields a section
\[
\wt{\iota}_\alpha: \Sc_X \to \Sc, (\sigma',T') \mapsto (\iota_\alpha(\sigma'),\iota_\alpha(T'))
\]
to $\wt{\rho}_X$ and provides a simple argument for the following lemma.
\begin{lemma}
	\label{lem:Kpi1LocOMS}
	If $\Sc$ is a $K(\pi,1)$ complex then so is $\Sc_X = \Sc(\OM_X)$ for all $X \in L(\OM)$.
\end{lemma}
\begin{proof}
	Application of the homotopy group functors 
	directly yields that the map $\pi_i(\wt{\rho}_X):\pi_i(\Sc) \to \pi_i(\Sc_X)$ has
	a section $\pi_i(\wt{\iota}_\alpha)$ and thus is surjective for all $i$.
\end{proof}

The realizable analogue of the previous result was first observed by Oka for complements of complex hyperplane arrangements,
cf.\ \cite[Lem.~1.1]{Paris1993_DeligneCpx}.
Corollary~\ref{coro:ModFlatCork1Kpi1} gives an inverse to Lemma~\ref{lem:Kpi1LocOMS} 
for a modular flat $X$ of corank one.

The subcomplex of $\Sc$ consisting of all the faces of a maximal cell can be identified with the dual covector complex.
\begin{lemma}
	\label{lem:SalPrincipialIdealCovectors}
	Let $(\Zero,T) \in \Sc$ be a maximal element of $\Sc$.
	Then $\Sc_{\leq (\Zero,T)} \isom \Lc^\dual$.
\end{lemma}
\begin{proof}
	The maps $\Sc_{\leq (\Zero,T)} \to \Lc^\dual, (F,R) \mapsto F$ 
	and $\Lc^\dual \to \Sc_{\leq (\Zero,T)}, F \mapsto (F,F\circ T)$ are 
	mutually inverse and order preserving.
\end{proof}

\begin{figure}
	\begin{tikzcd}[column sep=16mm, scale cd=1]
		\Sc \ar[r,"\wt{\rho}_X"] & \Sc_X \\
		\Sc_{\leq (\Zero,T)} \ar[u,hook] \ar[r,"\wt{\rho}_X|"]\ar[d,"\isom"] 
		&(\Sc_X)_{\leq(\Zero,\rho_X(T))} \ar[u,hook]\ar[d,"\isom"]  \\
		\Lc^\dual \ar[r,"\rho_X"] 
		&\Lc_X^\dual  
	\end{tikzcd}
	\caption{Localization maps restricted to subcomplexes of maximal cells are localization maps of the dual covector complexes.}
	\label{fig:DiagramSubcpxes} 
\end{figure}

\begin{corollary}
	\label{coro:SalRestrictionToPrincipalIdea}
	The diagram shown in Figure~\ref{fig:DiagramSubcpxes} commutes.		
\end{corollary}


\subsection{Discrete Morse theory}
\label{ssec:DiscMorseTheory}

We review some basic concepts from Forman's discrete Morse theory \cite{Forman1998_DiscrMorse} from the viewpoint
of acyclic matchings on posets respectively regular cell complexes adopted by Chari \cite{Chari2000_DiscreteMorseCombDecomp}. 
For a comprehensive exposition of the theory we refer to 
Kozlov's book \cite{Kozlov2008_CombAlgTop}.

\begin{definition}
	Let $P = (P,\leq)$ be a poset.
	We associate a directed graph $G(\Sigma) = (\Vc,\Ec)$ with vertex set $\Vc = P$ the ground set
	of $P$ and directed edges $\Ec := \{ (a,b) \mid a \lessdot b\}$ its cover relations,
	i.e. the Hasse diagram of the poset.
	
	A subset $\mtc \subseteq \Ec$ is called a \emph{matching} on $P$ if every element of $P$ appears
	in at most one $(a,b) \in \mtc$.
	
	Let $G(P,\mtc) = (\Vc,\Ec')$ be the new directed graph with $\Ec' := \Ec\setminus\mtc \cup \{(b,a) \mid (a,b) \in \mtc\}$, 
	i.e.\ the directed graph constructed from $G(P)$ by inverting all edges in $\mtc$.
	A matching $\mtc$ on $\Sigma$ is called \emph{acyclic} if $G(P,\mtc)$ has no directed cycles.
	
	For an acyclic matching $\mtc$ on $P$ its \emph{critical elements} $C(\mtc) \subseteq P$ are defined
	as
	\[
		C(\mtc) := \{ a \in P \mid a \notin e\text{ for all }e \in \mtc\}.
	\]
	
	In the same way we define acyclic matchings for a regular cell complex $\Sigma$ (which is identified with its poset of faces).
\end{definition}

Note that each acyclic matching $\mtc$ on a poset $P$ gives an acyclic matching 
$\mtc^\dual$ on the dual poset $P^\dual$
with $C(\mtc) = C(\mtc^\dual)$.

We record the following special case of the main theorem of discrete Morse theory,
cf.\ \cite{Forman1998_DiscrMorse}, \cite{Chari2000_DiscreteMorseCombDecomp}, \cite[Ch.~11]{Kozlov2008_CombAlgTop}.

\begin{theorem}
	\label{thm:MainThmDMTRegularSubcomplex}
	Let $\Sigma$ be a regular cell complex and let $\Gamma \subseteq \Sigma$ be a subcomplex.
	Assume that $\mtc$ is an acyclic matching on $\Sigma$ with $C(\mtc) = \Gamma$.
	Then $\Gamma$ is a strong deformation retract of $\Sigma$.
	In particular, the inclusion $\Gamma \hookrightarrow \Sigma$ is a homotopy equivalence.
\end{theorem}

The next proposition, due to Chari \cite{Chari2000_DiscreteMorseCombDecomp}, connects 
shellability with discrete Morse theory.
Another, more explicit description of acyclic matchings for
shellable regular cell complexes is due to Delucchi \cite{Delucchi2008_ShellableSalvetti}.

\begin{proposition}[{\cite[Prop.~4.1]{Chari2000_DiscreteMorseCombDecomp}}]
	\label{propo:ShellingMatchingBall}
	Let $\Sigma$ be a regular cell decomposition of the $d$-ball with shelling $\sigma_1,\ldots,\sigma_t$.
	Further, let $v \in \sigma_1$ be a vertex.
	Then there is an acyclic matching $\mtc$ on $\Sigma$ with $C(\mtc) = \{v\}$, i.e.\
	$\Sigma$ collapses onto $v$.
\end{proposition}

We conclude by recalling the following useful tool to construct acyclic matchings on posets, and hence on cell complexes.
\begin{theorem}[Patchwork theorem {\cite[Thm.~11.10]{Kozlov2008_CombAlgTop}}]
	\label{thm:Patchwork}
	Let $f:P \to Q$ be a poset map. Assume that for all $q \in Q$ we have an acyclic matching $\mtc(q)$
	on $f^{-1}(q)$. Then $\mtc = \bigcup\limits_{q \in Q}\mtc(q)$ is an acyclic matching on $P$.
\end{theorem}


\section{Combinatorics of modular flats of oriented matroids}
\label{sec:CombModularOMS}

\begin{figure}
	\begin{center}
		\begin{tikzcd}[column sep=15mm, scale cd=1]
			\Lc^\dual \ar[r, "\rho_X"] \ar[d, "z"] & \Lc_X^\dual \ar[d, "z"] 
			& \sigma \ar[r, "\rho_X", maps to] \ar[d, "z", maps to] & \sigma|_X \ar[d, "z", maps to] \\
			L \ar[r, "p_X"] & L_X & z(\sigma) \ar[r, "p_X", maps to] & z(\sigma)\wedge X = z(\sigma|_X)
		\end{tikzcd}
	\end{center}
	\caption{Restriction maps of face posets and corresponding geometric lattices.}
	\label{fig:DiagramFaceLatticeGeomLocal}
\end{figure}

We study the properties of modular flats of oriented matroids in more detail.
The basic idea is to use the commutative square of poset maps depicted in Figure~\ref{fig:DiagramFaceLatticeGeomLocal}
to derive refined statements about covector posets from the corresponding statements about geometric lattices.

Let $L = L(\OM)$ be the geometric lattice of a simple oriented matroid and $\Lc$ its covector poset.
Consider the localization map $\rho_X:\Lc^\dual \to \Lc_X^\dual$ for the dual covector complexes
for an $X \in L$.
We state a lemma which connects the notion of a modular flat of corank one
to special subsets of topes.
\begin{lemma}
	\label{lem:ModCoRk1LinExtFiber}
	Let $X \in L$ be modular of corank $1$ and $B' \in \Tc_X$.
	Then there are exactly two topes $B_1,B_2 \in (\rho_X\downarrow B')$
	such that the restricted poset $\Tc(\OM,B_i)|_{\Tc(\rho_X\downarrow B')}$ is linearly ordered.
	
	In particular, for each $T \in \Tc(\OM,B_i)|_{\Tc(\rho_X\downarrow B')}, T \neq B_i$
	there is a unique $T'\in \Tc(\rho\downarrow B')$ with $T'\lessdot T$ and 
	$S(T',T) = \{e\}$ such that $e \in S(R,T)$ for each $R < T$, $R \in \Tc(\OM,B_i)$. 
\end{lemma}
\begin{proof}
	Since $X$ has corank $1$, there are exactly two covectors in $z^{-1}(X)$, say $\alpha$
	and $\beta = -\alpha$. Let $\iota_\alpha,\iota_\beta$ be the associated sections to $\rho_X$
	from Remark~\ref{rem:SectionRhoX}.
	Set $B_1 := \iota_\alpha(B')$ and $B_2 := \iota_\beta(B')$.
	Then the modularity of $X$ forces every cell in $(\rho_X\downarrow B')$ to have dimension $\leq 1$.
	Otherwise, there would be a cell $\tau \in (\rho_X\downarrow B')$ with $\rho_X(\tau) = B'$ 
	and $Y = z(\tau) \in L$ a flat of rank $2$ with $Y \wedge X =z(\rho_X(\tau)) = z(B')= \emptyset$, 
	contradicting the modularity of $X$. 
	It follows that $(\rho_X\downarrow B')$ is a string with end vertices $B_1,B_2$. 
	Thus, the two topes $B_1,B_2$ have the desired property which yields both statements.	
\end{proof}

\begin{figure}
	\def\ha{(-0.6, 0.6) circle [radius=2]}
	\def\hb{(0.6, 0.6) circle [radius=2]}
	\def\hc{(0.6, -0.6) circle [radius=2]}
	\def\hd{(-0.6, -0.6) circle [radius=2]}
	\def\he{(-3.5,0) -- (3.75,0)}
	\def\hesq{(-4,0) -- (4,0) -- (4,4) -- (-4,4) -- (-4,0)}
	\begin{tikzpicture}[scale=0.85]
		
		
		\begin{scope}
			\clip (-0.6, 0.6) circle [radius=2];
			\clip (-4,0) -- (4,0) -- (4,4) -- (-4,4) -- (-4,0);
			
			\fill[black!10!white] (-0.6,-0.6) circle[radius=2]; 
		\end{scope}
			%
			%
		
		\draw[line width=1pt] (-0.6, 0.6) circle [radius=2];	
		\draw (0.6, 0.6) circle [radius=2];	
		\draw (0.6, -0.6) circle [radius=2];	
		\draw[line width=1pt] (-0.6, -0.6) circle [radius=2];	
		\draw[line width=1pt] (-3.25,0) -- (3.25,0);
		
		\filldraw (1.3,0) circle [radius=0.07];
		\filldraw (-2.5,0) circle [radius=0.07];
		
		\node[above left] at (-2.5,0) {$\alpha$};
		\node[above right] at (1.3,0) {$\beta$};

		\node (T0) at (0,0.5) {\small{$T_2$}};
		\node (T1) at (-1,1) {\small{$T_1$}};
		\node (T2) at (-1.75,0.5) {\small{$T_0$}};

		\node (v) at (8,0) {};;
		\draw[->] (3.7,0) -- (4.7,0); 
		\node[above] at (4.25,0) {$\rho_X$};
		
		\begin{scope}
			\clip ($(0, 0.6)+(v)$) circle [radius=2];
			\clip ($(-4,0)+(v)$) -- ($(4,0)+(v)$) -- ($(4,4)+(v)$) -- ($(-4,4)+(v)$) -- ($(-4,0)+(v)$);
			
			\fill[black!10!white] ($(0,-0.6)+(v)$) circle[radius=2]; 
		\end{scope}
		
		\draw[line width=1pt] ($(0, 0.6)+(v)$) circle [radius=2];	
		\draw[line width=1pt] ($(0, -0.6)+(v)$) circle [radius=2];	
		\draw[line width=1pt] ($(-3,0)+(v)$) -- ($(3,0)+(v)$);
		
		\node at ($(0,0.6)+(v)$) {$B'$};
		
		\node[right] at (2,2) {$H_4$};
		\node[left] at (-2.1,2) {$H_2$};
		\node[left] at (-2,-2) {$H_3$};
		\node[right] at (2,-2) {$H_5$};
		\node[below] at (3,0) {$H_1$};
		
	\end{tikzpicture}

	\caption{A linear order $\Tc(\rho_X\downarrow B') = \{T_0,T_1,T_2\}$ as described in Lemma \ref{lem:ModCoRk1LinExtFiber}.}
	\label{fig:LinExtFiber}
\end{figure}
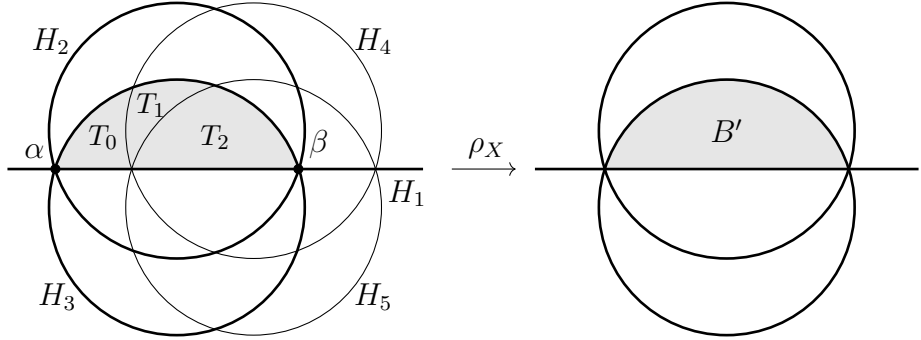

We give an example to illustrate Lemma \ref{lem:ModCoRk1LinExtFiber}.
\begin{example}
	\label{ex:LemLinExtTopesFiber}
	Consider the hyperplane arrangement $\Ac$ in $\RR^3$ consisting of the hyperplanes
	\begin{align*}
		\Ac = \{ &H_1=\ker(x),H_2=\ker(y),H_3=\ker(x+y),\\
		&H_4=\ker(z),H_5=\ker(x+z)\}		
	\end{align*}
	and $\OM = \OM(\Ac)$ the associated oriented matroid with topes $\Tc$.
	Then the flat $X = \{H_1,H_2,H_3\} \in L(\OM)$ is modular.
	A stereographic projection of $\Ac$ (or rather $\Ac \cap S^2$) is shown on the left hand side of Figure \ref{fig:LinExtFiber}.
	On the right hand side of Figure \ref{fig:LinExtFiber} we see its localization at $X$.
	
	Then the topes in the fiber of $B' \in \Tc_X$ are $\Tc(\rho_X\downarrow B') = \{T_0,T_1,T_2\}$
	with $T_0 = \iota_\alpha(B')$ and $T_2 = \iota_\beta(B')$.
	Indeed, $\Tc(\OM,T_0)|_{\Tc(\rho_X\downarrow B')}$ is linearly ordered as $T_0 < T_1 < T_2$.
	Furthermore, $S(T_1,T_2) = \{H_5\} \subseteq S(T_0,T_2) = \{H_4,H_5\}$.
	So $H_5 \in \Ac$ satisfies the second statement of Lemma \ref{lem:ModCoRk1LinExtFiber}.
\end{example}

Now, we consider the primal covector complex $\Lc$.
\begin{lemma}
	\label{lem:ModEltFacesExtension}
	Let $X \in L$ be a modular flat and $\sigma, \tau \in \Lc$ with $z(\sigma)=z(\tau)=Y \in L$
	and $\sigma|_X = \tau|_X$.
	Then $\sigma|_{X\vee Y} = \tau|_{X\vee Y}$.
\end{lemma}
\begin{proof}
	Suppose $\sigma|_{X\vee Y} \neq \tau|_{X\vee Y}$, i.e.\ there is an $e \in X \vee Y$ such that
	$\sigma_e \neq \tau_e \neq 0$. So $e \in S(\sigma, \tau) \cap X\vee Y$.
	By the fourth covector axiom (see Definition~\ref{def:OMAxioms}) there is a $\gamma \in \Lc$ with $\gamma_e = 0$ and
	$\gamma_f = (\sigma\circ \tau)_f = (\tau\circ \sigma)_f = \sigma_f = \tau_f$ for all $f \in X \vee Y \setminus S(\sigma,\tau)$.
	In particular $\gamma|_X = \tau|_X = \sigma|_X$ and $z(\sigma) = Y \subseteq z(\gamma)$. Set $Z := z(\gamma)$.
	Then $Z \in [Y, X\vee Y]$ with $Z > Y$. But we also have $Z \wedge X = Y \wedge X$ by construction.
	Thanks to Lemma~\ref{lem:IntervalsModGeomL}, we get $Z = (Z \wedge X) \vee Y = (Y \wedge X) \vee Y = Y$,
	contradicting $Z > Y$.
\end{proof}

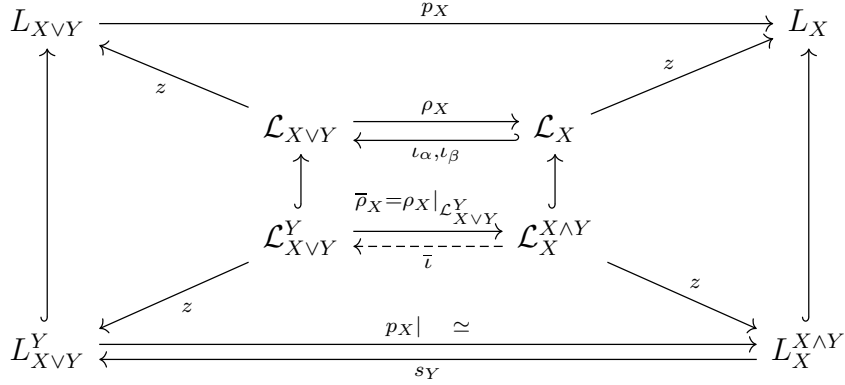
\begin{figure}
	\begin{center}
		\begin{tikzcd}[column sep=20mm, scale cd=1]
%

			L_{X\vee Y} \ar[rrr,"p_X"] & & & L_X \\
			&\Lc_{X\vee Y} \ar[lu, "z"] \ar[r,"\rho_X", yshift=0.1cm]  &\Lc_X \ar[l,"{\iota_{\alpha},\iota_{\beta}}", hook', yshift=-0.15cm] \ar[ru,"z"] &  \\
			&\Lc^Y_{X\vee Y} \ar[r,"{\overline{\rho}_X=\rho_X|_{\Lc^Y_{X\vee Y}}}",yshift=0.1cm] \ar[u, hook] \ar[ld, "z"] &\Lc_X^{X \wedge Y} \ar[l,"\overline{\iota}", yshift=-0.1cm, dashed] \ar[u, hook] \ar[rd, "z"] & \\
			L^Y_{X\vee Y} \ar[uuu, hook] \ar[rrr, "{p_X| \quad \isom}", yshift=0.1cm] & & & L_X^{X\wedge Y} \ar[lll,"s_Y", yshift=-0.1cm] \ar[uuu, hook]
		\end{tikzcd}
	\end{center}
	\caption{Proof of Lemma~\ref{lem:ModEltIntervalsCovectors}.}
	\label{fig:DiagramFaceLatticeGeomLocal_Iso}
\end{figure}

\begin{proposition}
	\label{prop:ModEltCovecRestrictFibers}
	Let $X,Y \in L(\OM) = L$ and let $\sigma \in \Lc_X^{X \wedge Y} \subseteq \Lc_X$.
	Then for a minimal element $\wh{\sigma} \in \rho_X^{-1}(\sigma) \cap \Lc^Y$
	we have $z(\wh{\sigma}) = z(\sigma)\vee_L Y$.
\end{proposition}
\begin{proof}
	If  $Y \subseteq X$ then the statement is trivial.
	So assume $Y \nsubseteq X$.
	Let $\alpha \in \Lc$ with $z(\alpha) = X$ and set $\beta = -\alpha \in \Lc$.
	The maps $\iota_\alpha,\iota_\beta: \Lc_X \to \Lc$ 
	from Remark~\ref{rem:SectionRhoX} give sections to $\rho_X$.
	Note that for all $\sigma \in \Lc_X$ we have 
	$S(\iota_\alpha(\sigma),\iota_\beta(\sigma)) = S(\alpha,\beta) = E \setminus X$.
	Since $Y \setminus X = Y \setminus X\wedge Y \subseteq S(\alpha,\beta)$, 
	the oriented matroid axiom (4) in Definition~\ref{def:OMAxioms} implies the statement.
\end{proof}

With the help of the previous proposition and Lemma~\ref{lem:ModEltFacesExtension} we can give the following refinement of Lemma~\ref{lem:IntervalsModGeomL} 
to covector posets of oriented matroids.

\begin{lemma}
	\label{lem:ModEltIntervalsCovectors}
	Let $X \in L(\OM)$ be modular and $Y \in L(\OM)$ another flat.
	Then the map
	\begin{align*}
		\overline{\rho}_X := \rho_X|_{\Lc^Y_{X\vee Y}}: \Lc_{X\vee Y}^Y &\to \Lc_X^{X\wedge Y}, \\
		\sigma &\mapsto \sigma|_X
	\end{align*}
	is an isomorphism of posets.
	
\end{lemma}
\begin{proof}
	We may assume without loss of generality that $X\vee Y = E$.
	Firstly, we confirm injectivity.
	Let $\sigma_1,\sigma_2 \in \Lc^Y$ with $\rho_X(\sigma_1)= \sigma_1|_X = \rho_X(\sigma_2)= \sigma_2|_X$.
	By the modularity of $X$ and Lemma~\ref{lem:IntervalsModGeomL} we have that 
	$z(\sigma_i) = Y\vee (z(\sigma_i) \wedge X) = Y \vee z(\rho_X(\sigma_i))$ for $i=1,2$. So $z(\sigma_1) = z(\sigma_2)$
	and $\sigma_1 = \sigma_1|_{X\vee Y} = \sigma_2|_{X\vee Y} = \sigma_2$ by Lemma~\ref{lem:ModEltFacesExtension}.
	Consequently, $\overline{\rho}_X$ is injective.
	
	Secondly, we construct a section $\overline{\iota}:\Lc_X^{X\wedge Y} \to \Lc_{X\vee Y}^Y$ 
	to $\overline{\rho}_X$ which implies the statement.
	By Proposition~\ref{prop:ModEltCovecRestrictFibers} and the commutativity of
	the diagram in Figure~\ref{fig:DiagramFaceLatticeGeomLocal_Iso}, for $\sigma \in \Lc_X^{X \wedge Y}$ 
	we have that 
	\[
	z(\rho_X^{-1}((\Lc_X)_{\geq \sigma}) \cap \Lc^Y) = L^Y_{\leq s_Y(z(\sigma))}.
	\]
	So there is a $\wh{\sigma} \in \rho_X^{-1}((\Lc_X)_{\geq \sigma}) \cap \Lc^Y$ with 
	$z(\wh{\sigma}) = s_Y(z(\sigma)) = Y \vee z(\sigma)$
	and further,  by Lemma~\ref{lem:IntervalsModGeomL}, $z(\overline{\rho}_X(\wh{\sigma})) = z(\sigma)$.
	This implies $\overline{\rho}_X(\wh{\sigma}) = \sigma$. 
	Moreover, by the injectivity of $\overline{\rho}_X$,
	this minimal $\wh{\sigma}$ is the unique minimal element of $\rho_X^{-1}((\Lc_X)_{\geq \sigma}) \cap \Lc^Y$.
	For $\sigma \leq \tau \in \Lc_X^{X\wedge Y}$
	we have $\rho_X^{-1}((\Lc_X)_{\geq \sigma}) \cap \Lc^Y \supseteq \rho_X^{-1}((\Lc_X)_{\geq \tau}) \cap \Lc^Y$
	and thus $\wh{\sigma} \leq \wh{\tau}$ for the unique minimal elements in these order filter of $\Lc^Y$.
	This yields our desired order preserving section $\overline{\iota}$ 
	by setting $\overline{\iota}(\sigma) := \wh{\sigma}$.
\end{proof}

\begin{figure}	
	
	\begin{tikzpicture}[scale=0.6]
		\node (1) at (-1.,3.,0.) {};
		\node (2) at (1.,-3.,0.)  {};
		\node (3) at (-3.,3.,0.) {};
		\node (4) at (3.,-3.,0.) {};
		\node (5) at (1.,3.,-2.) {}; 
		\node (6) at (-1.,-3.,2.) {};
		\node (7) at (-1.,1.,2.) {};
		\node (8) at (1.,-1.,-2.) {};
		\node (9) at (-3.,1.,0.) {};
		\node (10) at (3.,-1.,0.) {};
		\node (11) at (-1.,3.,-2.) {};
		\node (12) at (1.,-3.,2.) {};
		\node (13) at (-3.,1.,2.) {};
		\node (14) at (3.,-1.,-2.) {};
		\node (15) at (3.,1.,-2.) {};
		\node (16) at (-3.,-1.,2.) {}; 
		\node (17) at (1.,-1.,2.) {};
		\node (18) at (-1.,1.,-2.) {};

		\fill[black!20!white] ($(1)$) -- ($(3)$) -- ($(13)$)-- ($(7)$) -- ($(1)$);
		\fill[black!20!white] ($(3)$) -- ($(9)$) -- ($(16)$)-- ($(13)$) -- ($(3)$);
		
		\draw[line width=2pt, black!20!white] ($(1)$) -- ($(3)$) -- ($(13)$)-- ($(7)$) -- ($(1)$);
		\draw[line width=2pt, black!20!white] ($(9)$) -- ($(16)$);
		\draw[line width=2pt, black!20!white] ($(13)$) -- ($(16)$);
		
		\filldraw[black!20!white] (1) circle[radius=5pt];
		\filldraw[black!20!white] (3) circle[radius=5pt];
		\filldraw[black!20!white] (13) circle[radius=5pt];
		\filldraw[black!20!white] (7) circle[radius=5pt];
		\filldraw[black!20!white] (9) circle[radius=5pt];
		\filldraw[black!20!white] (16) circle[radius=5pt];
		
		\node[left, black!60!white] at ($0.5*(3)+0.5*(13)$) {$(\rho_X\downarrow\sigma)$};
		
		
		\filldraw (1) circle[radius=2pt];
		\filldraw (2) circle[radius=2pt];
		\filldraw (3) circle[radius=2pt];
		\filldraw (2) circle[radius=2pt];
		\filldraw (5) circle[radius=2pt];
		\filldraw (6) circle[radius=2pt];
		\filldraw (7) circle[radius=2pt];
		\filldraw (8) circle[radius=2pt];
		\filldraw (9) circle[radius=2pt];
		\filldraw (10) circle[radius=2pt];
		\filldraw (11) circle[radius=2pt];
		\filldraw (12) circle[radius=2pt];
		\filldraw (13) circle[radius=2pt];
		\filldraw (14) circle[radius=2pt];
		\filldraw (15) circle[radius=2pt];
		\filldraw (16) circle[radius=2pt];
		\filldraw (17) circle[radius=2pt];
		\filldraw (18) circle[radius=2pt];
		
		
		\draw [shorten <=0.125pt, shorten >=0.125pt] (-1.,3.,0.) -- %
		(-3.,3.,0.); 
		\draw [shorten <=0.125pt, shorten >=0.125pt] (-1.,3.,0.) -- %
		(1.,3.,-2.); 
		\draw [shorten <=0.125pt, shorten >=0.125pt] (-1.,3.,0.) -- %
		(-1.,1.,2.); 
		\draw [shorten <=0.125pt, shorten >=0.125pt] (1.,-3.,0.) -- %
		(3.,-3.,0.); 
		\draw [shorten <=0.125pt, shorten >=0.125pt] (1.,-3.,0.) -- %
		(-1.,-3.,2.); 
		\draw [shorten <=0.125pt, shorten >=0.125pt] (1.,-3.,0.) -- %
		(1.,-1.,-2.); 
		\draw [shorten <=0.125pt, shorten >=0.125pt] (-3.,3.,0.) -- %
		(-3.,1.,0.); 
		\draw [shorten <=0.125pt, shorten >=0.125pt] (-3.,3.,0.) -- %
		(-1.,3.,-2.); 
		\draw [shorten <=0.125pt, shorten >=0.125pt] (-3.,3.,0.) -- %
		(-3.,1.,2.); 
		\draw [shorten <=0.125pt, shorten >=0.125pt] (3.,-3.,0.) -- %
		(3.,-1.,0.); 
		\draw [shorten <=0.125pt, shorten >=0.125pt] (3.,-3.,0.) -- %
		(1.,-3.,2.); 
		\draw [shorten <=0.125pt, shorten >=0.125pt] (3.,-3.,0.) -- %
		(3.,-1.,-2.); 
		\draw [shorten <=0.125pt, shorten >=0.125pt] (1.,3.,-2.) -- %
		(-1.,3.,-2.); 
		\draw [shorten <=0.125pt, shorten >=0.125pt] (1.,3.,-2.) -- %
		(3.,1.,-2.); 
		\draw [shorten <=0.125pt, shorten >=0.125pt] (-1.,-3.,2.) -- %
		(1.,-3.,2.); 
		\draw [shorten <=0.125pt, shorten >=0.125pt] (-1.,-3.,2.) -- %
		(-3.,-1.,2.); 
		\draw [shorten <=0.125pt, shorten >=0.125pt] (-1.,1.,2.) -- %
		(-3.,1.,2.); 
		\draw [shorten <=0.125pt, shorten >=0.125pt] (-1.,1.,2.) -- %
		(1.,-1.,2.); 
		\draw [shorten <=0.125pt, shorten >=0.125pt] (1.,-1.,-2.) -- %
		(3.,-1.,-2.); 
		\draw [shorten <=0.125pt, shorten >=0.125pt] (1.,-1.,-2.) -- %
		(-1.,1.,-2.); 
		\draw [shorten <=0.125pt, shorten >=0.125pt] (-3.,1.,0.) -- %
		(-3.,-1.,2.); 
		\draw [shorten <=0.125pt, shorten >=0.125pt] (-3.,1.,0.) -- %
		(-1.,1.,-2.); 
		\draw [shorten <=0.125pt, shorten >=0.125pt] (3.,-1.,0.) -- %
		(3.,1.,-2.); 
		\draw [shorten <=0.125pt, shorten >=0.125pt] (3.,-1.,0.) -- %
		(1.,-1.,2.); 
		\draw [shorten <=0.125pt, shorten >=0.125pt] (-1.,3.,-2.) -- %
		(-1.,1.,-2.); 
		\draw [shorten <=0.125pt, shorten >=0.125pt] (1.,-3.,2.) -- %
		(1.,-1.,2.); 
		\draw [shorten <=0.125pt, shorten >=0.125pt] (-3.,1.,2.) -- %
		(-3.,-1.,2.); 
		\draw [shorten <=0.125pt, shorten >=0.125pt] (3.,-1.,-2.) -- %
		(3.,1.,-2.); 
		
		\node at (0,-4.5) {$\Lc^\dual \supseteq (\Lc^H)^\dual$};
		
		\node (m1) at ($0.5*(7)+0.5*(13)$) {};
		\node (m2) at ($0.5*(6)+0.5*(12)$) {};
		\node (m3) at ($0.5*(2)+0.5*(4)$) {};
		\node (m4) at ($0.5*(8)+0.5*(14)$) {};
		\node (m5) at ($0.5*(5)+0.5*(11)$) {};
		\node (m6) at ($0.5*(1)+0.5*(3)$) {};
		\filldraw (m1) circle[radius=3pt];
		\filldraw (m2) circle[radius=3pt];
		\filldraw (m3) circle[radius=3pt];
		\filldraw (m4) circle[radius=3pt];
		\filldraw (m5) circle[radius=3pt];
		\filldraw (m6) circle[radius=3pt];
		
		\draw[line width=2pt, dotted] (m1) -- (m2) -- (m3) -- (m4) -- (m5) -- (m6) -- (m1);

		
		\node (ar) at (5,0,0) {};
		\draw[->] ($(ar)-(0.5,0,0)$) -- ($(ar)+(0.5,0,0)$);
		\node[above] at (ar) {$\rho_X$};
		
		\node (v) at (8.5,0,0) {};;
		
		\node (l1) at (1.7320508075688774,-1.,0) {};
		\node (l2) at (1.7320508075688774,1.,0) {};
		\node (l3) at (0.0,2.,0) {};
		\node (l4) at (-1.7320508075688774,1.,0) {};
		\node (l5) at (-1.7320508075688774,-1.,0) {};
		\node (l6) at (0.0,-2.,0) {};
		
		\draw ($1.25*(l1)+(v)$) -- ($1.25*(l2)+(v)$) -- ($1.25*(l3)+(v)$) -- ($1.25*(l4)+(v)$) -- ($1.25*(l5)+(v)$) -- ($1.25*(l6)+(v)$) -- ($1.25*(l1)+(v)$);
		
		\draw[black!30!white, line width=2pt] ($1.25*(l3)+(v)$) -- ($1.25*(l4)+(v)$);
		
		\node[above left, black!60!white] at ($0.625*(l3)+0.625*(l4)+(v)$) {$\sigma$};

		\filldraw ($1.25*(l1)+(v)$) circle[radius=2pt];
		\filldraw ($1.25*(l2)+(v)$) circle[radius=2pt];
		\filldraw ($1.25*(l3)+(v)$) circle[radius=2pt];
		\filldraw ($1.25*(l4)+(v)$) circle[radius=2pt];
		\filldraw ($1.25*(l5)+(v)$) circle[radius=2pt];
		\filldraw ($1.25*(l6)+(v)$) circle[radius=2pt];


		\node at (8.5,-4.5) {$\Lc_X^\dual$};
		
	\end{tikzpicture}
	\caption{The dual covector complex $\Lc^\dual = \Lc^\dual(\Ac)$ of the arrangement $\Ac$, and its contraction $(\Lc^H)^\dual$, 
		which is mapped isomorphically to the localization $\Lc_X^\dual$ at a modular flat $X$.}
	\label{fig:ModEltCovectorsIso}
\end{figure}
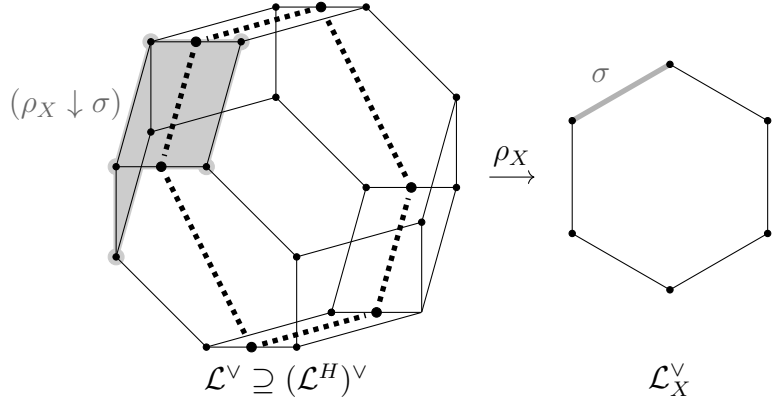

We visualize Lemma~\ref{lem:ModEltIntervalsCovectors} with an example.
\begin{example}
	\label{ex:ModEltCovectorsIso}
	Let $\Ac$ be the hyperplane arrangement in $\RR^3$ consisting of the hyperplanes
	\begin{align*}
		\Ac = \{ &H_1=\ker(x),H_2=\ker(y),H_3=\ker(x+y),\\
		&H_4=\ker(z),H_5=\ker(x+z)\}		
	\end{align*}
	and $\OM = \OM(\Ac)$ the associated oriented matroid with covectors $\Lc = \Lc(\Ac)$.
	Then the flat $X = \{H_1,H_2,H_3\} \in L(\OM)$ is modular.
	The dual covector complex $\Lc^\dual$, the localization $\Lc_X^\dual$ and the map $\rho_X:\Lc^\dual \to \Lc_X^\dual$ 
	together with a poset fiber $(\rho_X\downarrow \sigma)$ are shown in Figure~\ref{fig:ModEltCovectorsIso}.
	The dual of the restricted covector complex $(\Lc^H)^\dual$ to $H=\{H_4\} \in L(\OM)$
	is given by the dotted subcomplex transversely to the cells in $z^{-1}(L^H) \isom (\Lc^H)^\dual$.
	We clearly see that the localization map $\rho_X$, restricted to $(\Lc^H)^\dual$ is an isomorphism
	in accordance with Lemma~\ref{lem:ModEltIntervalsCovectors}.
\end{example}


\section{A stratification of subcomplexes of the Salvetti complex}
\label{sec:StratSalvettiFibers}

Again, let $X \in L$ be a modular flat of corank $1$.
We now describe a special stratification of the poset fiber $(\wt{\rho}_X\downarrow(\Zero,B'))$ for $B' \in \Tc_X$.
This is analogous to a result by Delucchi \cite[Lem.~4.17,Lem.~4.18]{Delucchi2008_ShellableSalvetti}.
Unfortunately, Delucchi's more general statement is not correct, cf.\ \cite{DelucchiMuecksch2022_Erraturm-ShellableSalvetti}.

Recall that by Lemma~\ref{lem:ModCoRk1LinExtFiber} we have a natural linear order on the subset of topes $\Tc(\rho_X\downarrow B')$,
say $\Tc(\rho_X\downarrow B') = \{T_0 < T_1 < ... < T_k\}$.
Then it also follows from Lemma~\ref{lem:ModCoRk1LinExtFiber} that the order filter
\[
J_i := \{X \in L \mid S(T_j,T_i)\cap X \neq \emptyset\text{ for all }j<i\} \subseteq L
\] 
is principal, i.e.\
\[
J_i = \begin{cases}
	L &\text{if } i=0,\\
	L^{S(T_{i-1},T_i)} &\text{ if } 1\leq i \leq k,
\end{cases}
\]
where $S(T_{i-1},T_i)$ consists of a single element in $E \setminus X$.

Now, we define locally closed subcomplexes $N_i \subseteq (\wt{\rho}_X\downarrow(\Zero,B'))$ by
\[
N_i := \Sc_{\leq (\Zero,T_i)} \setminus (\bigcup_{j<i} \Sc_{\leq (\Zero,T_j)}).
\]
Note that $(\wt{\rho}_X\downarrow(\Zero,B')) = \bigsqcup_{0\leq i \leq k}N_i$.

Furthermore, we have the following useful lemma.

\begin{lemma}
	\label{lem:WTRhoFiberStrat}
	For all $0\leq i \leq k$ we have $N_i \cong (\Lc^{S(T_{i-1},T_i)})^\dual$.
\end{lemma}
\begin{proof}
	This follows from the fact that $J_i$ is a principal order filter.
	See the proof of \cite[Lem.~4.18]{Delucchi2008_ShellableSalvetti} or \cite[Lem.~4.12]{LofanoPaolini2021_EuclideanMatchings}.
\end{proof}


\section{Poset quasi-fibrations}
\label{sec:QullenThmB}

We introduce a new class of poset maps with a desirable topological property.

Firstly, we record the following theorem from Quillen's foundational
work \cite{Quillen1973_KTheory1} on algebraic $K$-Theory, originally expressed in terms
of small categories and functors.
We state it in its special incarnation in terms of posets and order preserving maps.

\begin{theorem}[Theorem B for posets {\cite[pp.~89]{Quillen1973_KTheory1}}]
	\label{thm:TheoremB}
	Let $f:P \to Q$ be an order preserving map between posets.
	If for all $a \leq b$ $(a,b \in Q)$ the inclusion
	$(f\downarrow a) = f^{-1}(Q_{\leq a}) \hookrightarrow f^{-1}(Q_{\leq b}) = (f\downarrow b)$ is a homotopy equivalence,
	then for $x \in P$ with $f(x)=a$ the homotopy fiber of 
	$F(|\Delta(f)|,a)$ is homotopy equivalent to $|\Delta(f\downarrow a)|$.
	
	Consequently, we have a long exact homotopy sequence
	
	\begin{center}
		\begin{tikzcd}
			&\ldots \ar[r] &\pi_{i+1}(|\Delta(Q)|,a) \ar[r] &\pi_{i}(|\Delta(f\downarrow a)|,x) \\
			 &	\ar[r, "\pi_i(|\Delta(j)|)"] &\pi_{i}(|\Delta(P)|,x) \ar[r, "\pi_i(|\Delta(f)|)"] &\pi_{i}(|\Delta(Q)|,a) \ar[r] &\ldots
		\end{tikzcd}
	\end{center}
	where $j:(f\downarrow a) \hookrightarrow P$ is the natural inclusion.
%
%
%
\end{theorem}

Theorem \ref{thm:TheoremB} motivates the following notion.	
\begin{definition}
	Let $f:P \to Q$ be a poset map such that for all $a \leq b$ $(a,b \in Q)$
	the inclusion $(f\downarrow a) \hookrightarrow (f\downarrow b)$ 
	is a homotopy equivalence.
	Then $f$ is called a \emph{poset quasi-fibration}.
\end{definition}

Invoking a homotopy colimit construction \cite[Prop.~3.12 \& Rem.]{WZZ1999_HomotopyColimComb}, 
every poset quasi-fibration can be combinatorially ``turned into'' 
a proper quasi-fibration in the sense of Dold and Thom \cite{DoldThom1958_Quasifaserungen} (after simplicial realization) 
with fibers homotopy equivalent to the poset fibers of the original order preserving map.


\section{Acyclic matchings for subcomplexes and the main result}
\label{sec:MatchingsAndPosetQFibrations}

Let $\OM = (E,\Lc)$ be an oriented matroid and $L := L(\OM)$ its geometric lattice.

Firstly, we derive certain acyclic matchings on the dual covector complex $\Lc^\dual$.
Recall from Definition \ref{def:ConvexSubsets} the subcomplex $\Lc^\dual[\Qc]$ for a convex subset $\Qc \subseteq \Tc$ of topes.

\begin{proposition}
	\label{prop:MatchingsCovectorShellable}
	Let $\emptyset \neq \Qc \subseteq \Tc$ be a convex subset and $\Lc^\dual[\Qc] \subseteq \Lc^\dual$
	its associated dual subcomplex.
	Then there is an acyclic matching $\mtc$ on $\Lc^\dual$ with $C(\mtc) = \Lc^\dual[\Qc]$.
\end{proposition}
\begin{proof}
	By Corollary~\ref{coro:ShellableComplTConvex} the subcomplex $\Lc(\Tc\setminus\Qc)$ is a shellable ball.
	Hence, by Proposition~\ref{propo:ShellingMatchingBall} there is an acyclic matching $\mtc'$ on $\Lc(\Tc\setminus\Qc)$
	whose critical cells consist of only one vertex of $v \in \Lc(\Tc\setminus\Qc)$. 
	Since $\Lc(\Tc\setminus\Qc)$ is an order ideal in $\Lc$, the matching $\mtc'$ can be extended to a matching $\mtc$
	on $\Lc$ by a trivial matching on the order filter $\Lc \setminus \Lc(\Tc\setminus\Qc)$ (e.g.\ use Theorem~\ref{thm:Patchwork}).
	Hence, this new matching $\mtc$ has critical cells: $C(\mtc) = \Lc \setminus \Lc(\Tc\setminus\Qc) \sqcup \{v\}$.
	By Remark~\ref{rem:ConvexSubsetsDualCpx} we have $(\Lc \setminus \Lc(\Tc\setminus\Qc))^\dual = \Lc^\dual[\Qc]$. 
	We can extend the matching $\mtc^\dual$ on $\Lc^\dual$ by one further matched pair $(v,\Zero)$,
	since $\Zero \in \Lc^\dual$ is the unique maximal element of $\Lc^\dual$ and $v \lessdot_{\Lc^\dual} \Zero$.
	This gives the claimed matching on $\Lc^\dual$.
\end{proof}

Let $X \in L$ be some flat and $\rho_X:\Lc^\dual \to \Lc_X^\dual$ the associated localization map. 
From the preceding proposition we immediately get the following.

\begin{corollary}
	\label{coro:MatchingFiberRho}
	Let $\sigma \in \Lc_X^\dual$. Then there is an acyclic matching on $\Lc^\dual$ with $C(\mtc) = (\rho_X\downarrow\sigma)$.
\end{corollary}
\begin{proof}
	This readily follows from Proposition~\ref{prop:MatchingsCovectorShellable}, the convexity of the subset of topes $\Tc(\rho_X\downarrow\sigma)$
	thanks to Lemma~\ref{lem:FiberPosetTopesConvex} and the observation that we
	have $\Lc^\dual[\Tc(\rho_X\downarrow\sigma)] = (\rho_X\downarrow\sigma)$.
\end{proof}

\bigskip
\noindent
Now assume that $X$ is modular and of corank $1$.
Recall the map $\wt{\rho}_X:\Sc \to \Sc_X, (\sigma,T) \mapsto (\rho_X(\sigma),\rho_X(T))$
between the Salvetti complexes.
Next, we derive special acyclic matchings on poset fibers of $\wt{\rho}_X$ 
of maximal elements of $\Sc_X$ which yield a homotopy equivalence.

\begin{figure}
	\begin{tikzcd}
		\Sc \ar[r,"\wt{\rho}_X"] & \Sc_X \\
		N_i \ar[u,hook] \ar[r,"\wt{\rho}_X|_{N_i}"] \ar[d,"\isom"] 
			&(\Sc_X)_{\leq (\Zero,B')} \ar[u,hook] \ar[d,"\isom"] \\
		(\Lc^{S(T_{i-1},T_i)})^\dual \ar[r,"\rho_X|"] \ar[d,hook] & \Lc_X^\dual \\
		\Lc^\dual \ar[ur, "\rho_X"] 		
	\end{tikzcd}
	\caption{Restriction of $\wt{\rho}_X$ to $N_i$ can be identified with the restriction of $\rho_X$.}
	\label{fig:RestrictionWTRhoNi}
\end{figure}

\begin{proposition}
	\label{prop:AcyclicMatchingFibersSal}
	Let $a\in \Sc_X$ and let $B' \in \Tc_X$ with $a\leq(\Zero,B')$. 
	Then there is an acyclic matching $\mtc$ on $(\wt{\rho}\downarrow (\Zero,B'))$
	with $C(\mtc) = (\wt{\rho}\downarrow a)$.
	In particular, $(\wt{\rho}\downarrow a)$ is a strong deformation retract of $(\wt{\rho}\downarrow (\Zero,B'))$
	and $(\wt{\rho}\downarrow a) \hookrightarrow (\wt{\rho}\downarrow (\Zero,B'))$ is a homotopy equivalence.
\end{proposition}
\begin{proof}
	Firstly, let $T_0<T_1<\ldots<T_k$ be a linear order on the set of topes $\Tc(\rho_X\downarrow B')$ as described by Lemma~\ref{lem:ModCoRk1LinExtFiber}.
	We utilize the stratification of $(\wt{\rho}_X\downarrow (\Zero,B')) = \bigsqcup_{0\leq i \leq k}N_i$
	from Section~\ref{sec:StratSalvettiFibers}.
	Note that each $N_i \subseteq (\wt{\rho}_X\downarrow (\Zero,B'))$ is an order filter
	in $\Sc(\{(\Zero,T_0),\ldots,(\Zero,T_i)\})$. 
	Thus, we have a natural order preserving map 
	\[
		f:(\wt{\rho_X}\downarrow (\Zero,B')) \to \{T_0 < T_1<\ldots<T_k\}
	\]
	with $f^{-1}(T_i) = N_i$.
	
	Now, by Lemma~\ref{lem:WTRhoFiberStrat} and Lemma~\ref{lem:SalPrincipialIdealCovectors} 
	we have $N_0 = \Sc_{\leq (\Zero,T_0)} \isom \Lc^\dual$,
	and for each $1\leq i \leq k$ we have 
	$\Sc_{\leq (\Zero,T_i)} \supseteq N_i \cong (\Lc^{S(T_{i-1},T_i)})^\dual$.
	Hence, by Corollary~\ref{coro:SalRestrictionToPrincipalIdea}
	the map $\wt{\rho}_X$ restricted to $N_0$ is nothing else than the map $\rho_X:\Lc^\dual \to \Lc_X^\dual$.
	Similarly, for $1\leq i \leq k$ the map $\wt{\rho}_X$ restricted to $N_i$ is 
	$\wt{\rho}_X|_{N_0} = \rho_X|:{(\Lc^{S(T_{i-1},T_i)})^\dual} \to \Lc_X^\dual$, 
	which is an isomorphism by Lemma~\ref{lem:ModEltIntervalsCovectors}, since $S(T_{i-1},T_i) \cap X = \emptyset$;
	see the commutative diagram shown in Figure \ref{fig:RestrictionWTRhoNi}.	
	Assume that $a = (\sigma,R')$ for some $R' \in \Tc_X$.
	Then: 
	\begin{itemize}
		\item 
		$N_0 \isom (\rho_X\downarrow \Zero) \isom \Lc^\dual$
		$\supseteq$ 
		$(\wt{\rho}_X\downarrow a)\cap N_0 \isom (\rho_X\downarrow \sigma)$,
		
		\item 
		$N_i \isom (\rho_X|\downarrow \Zero) \isom (\Lc^{S(T_{i-1},T_i)})^\dual$
		$\supseteq$
		$(\wt{\rho}_X\downarrow a)\cap N_i \isom (\rho_X|\downarrow \sigma) \isom (\Lc_X^\dual)_{\leq \sigma}$,
		for $1\leq i \leq k$.
	\end{itemize}
	In both instances, by Corollary~\ref{coro:MatchingFiberRho} we get a matching $\mtc(i)$
	on $N_i$ with critical cells $C(\mtc(i)) = (\wt{\rho}_X\downarrow a)\cap N_i$.
	Patching these matchings together via Theorem~\ref{thm:Patchwork} along the map $f$ 
	yields our desired matching $\mtc = \bigcup_{i=0}^k\mtc(i)$
	on $(\wt{\rho}_X\downarrow (\Zero,B'))$ with $C(\mtc) = (\wt{\rho}_X\downarrow a)$
	and concludes the proof.	
\end{proof}

From the previous proposition we can now readily derive our principal result.

\begin{theorem}
	\label{thm:MainQuasifibration}
	Let $X \in L(\OM)$ be a modular flat of corank $1$. Then the map $\wt{\rho}_X:\Sc \to \Sc_X$
	is a poset quasi-fibration.
	For $a \in \Sc_X$ the poset fiber $(\wt{\rho}_X\downarrow a)$ is homotopy equivalent
	to the affine Salvetti complex of an oriented matroid $\Nc$ of rank $2$.
\end{theorem}
\begin{proof}
	Let $a, b \in \Sc_X$ with $a\leq b$. We have to show that the inclusion
	$(\wt{\rho_X}\downarrow a) \hookrightarrow (\wt{\rho}_X\downarrow b)$ is a homotopy equivalence.
	Picking some maximal element $(\Zero, B')$ of $\Sc_X$ with $b \leq (\Zero, B')$,
	we have a commutative triangle of inclusion maps:
	\begin{center}
	\begin{tikzcd}
		 & (\wt{\rho}_X\downarrow (\Zero,B')) & \\
		(\wt{\rho}_X\downarrow a) \ar[ur,hook, "i_1"] \ar[rr,hook,"i_3"] & & (\wt{\rho}_X\downarrow b). \ar[ul,hook',"i_2" above]
	\end{tikzcd}
	\end{center}
	
	By Proposition~\ref{prop:AcyclicMatchingFibersSal}
	the two inclusions $i_1$ and $i_2$ are homotopy equivalences.
	Hence, so is the third inclusion $i_3$.
	This confirms that $\wt{\rho}_X$ is a poset quasi-fibration.
	
	To see the last statement, pick a minimal element $(B',B') \in \Sc_X$. 
	We easily see that $(\rho_X\downarrow B') = \Nc_{\aff,g}$ for
	an oriented matroid $\Nc$ of rank two on $E \setminus X \sqcup \{g\}$ by Lemma \ref{lem:ModCoRk1LinExtFiber}.
	Then $(\wt{\rho}_X\downarrow (B',B')) = \Sc_{\aff,g}(\Nc)$.
\end{proof}

The affine Salvetti complex of an oriented matroid of rank $2$ is a graph
and consequently aspherical.
Hence, by the long exact sequence of homotopy groups associated to a poset quasi-fibration
due to Theorem~\ref{thm:TheoremB}
we immediately obtain the following.

\begin{corollary}
	\label{coro:ModFlatCork1Kpi1}
	If $X \in L(\OM)$ is modular of corank $1$ and $\Sc_X$ is aspherical,
	then $\Sc$ is aspherical as well.
\end{corollary}

\begin{corollary}[{Theorem~\ref{thm:SSOMKpi1}}]
	\label{coro:SSOMsKpi1}
	The Salvetti complex of a supersolvable oriented matroid is aspherical.
\end{corollary}


\section{Fundamental groups}
\label{sec:FundGrps}

We analyze the structure of the fundamental group of the Salvetti complex of a
supersolvable oriented matroid. 
This is completely analogous to the realizable case, cf. \cite{FalkRandell1985_FiberTypeArrs}.

Let $P$ be a poset and $x \in P$. Then we write $\pi_1(P,x)= \pi_1(|\Delta(P)|,x)$ for the fundamental group of $P$.
Similarly, if $f:P\to Q$ is an order preserving map, then we write 
$\pi_1(f) := \pi_1(|\Delta(f)|):\pi_1(|\Delta(P)|,x) \to \pi_1(|\Delta(Q)|,f(x))$ for the induced map
between fundamental groups. 

Firstly, we have the following lemma.
\begin{lemma}
	\label{lem:PosetQuasiFibrSectionPi1}
	Assume that $f:P \to Q$ is a poset quasi-fibration which has a section $s:Q\to P$.
	Let $x \in P$ and $a = f(x)$ and assume further that $|\Delta(Q)|$ and $|\Delta(f\downarrow a)|$ are $K(\pi,1)$-spaces.
	Then the fundamental group $\pi_1(P,x)$ is a semi-direct product:
	\[
		\pi_1(P,x) \cong \pi_1((f\downarrow a),x) \rtimes \pi_1(Q,a).
	\]
\end{lemma}
\begin{proof}
	By Theorem \ref{thm:TheoremB}, the long exact sequence of the poset quasi-fibration degenerates 
	to the short exact sequence:
	\begin{center}
		\begin{tikzcd}
			1 \ar[r] &\pi_1((f\downarrow a),x) \ar[r] &\pi_1(P,x) \ar[r, "\pi_1(f)"] &\pi_1(Q,a) \ar[r] &1.
		\end{tikzcd}
	\end{center}
	The map $\pi_1(s)$ gives a section to $\pi_1(f)$ and so the above short exact sequence splits.
\end{proof}

With the preceding lemma and our main result from the last section we can say the following.
\begin{theorem}
	\label{thm:Pi1ModCork1}
	Let $\OM = (E,\Lc)$ be an oriented matroid and assume that $X \in L(\OM)$ is modular and
	of corank $1$.
	Assume further that $\Sc_X$ is aspherical.
	Then for $x \in \Sc$ and $a = \wt{\rho}_X(x) \in \Sc_X$ we have
	\[
		\pi_1(\Sc,x) \cong F \rtimes \pi_1(\Sc_X,a),
	\]
	where $F$ is a free group with $|E \setminus X|$ generators.
\end{theorem}
\begin{proof}
	By Theorem \ref{thm:MainQuasifibration} the map $\wt{\rho}_X:\Sc \to \Sc_X$
	is a poset quasi fibration with $(\wt{\rho}_X\downarrow a)$ homotopy equivalent to
	the affine Salvetti complex of an oriented matroid of rank $2$ on $|E\setminus X|+1$ elements.
	Hence $(\wt{\rho}_X\downarrow a) \isom S^1\vee \ldots \vee S^1$ is a wedge of $|E\setminus X|$ circles.
	Thus, $\pi_1((\wt{\rho}_X\downarrow a),a) = F$ is a free group with $|E \setminus X|$ generators.
	Furthermore, $\wt{\rho_X}$ has a section by Remark \ref{rem:SectionRhoX} and 
	Lemma \ref{lem:PosetQuasiFibrSectionPi1} yields the result.
\end{proof}

Consequently, the structure of the fundamental group of the Salvetti complex of a supersolvable
oriented matroid is as follows.
\begin{corollary}
	\label{coro:Pi1SSOM}
	Let $\OM$ be a supersolvable oriented matroid of rank $r$.
	Then $\pi_1(\Sc,x) \cong F_1 \rtimes (F_2 \rtimes (\ldots (F_{r-1} \rtimes \ZZ)\ldots))$
	where the $F_i$ are finitely generated free groups.
\end{corollary}


\section{Non-realizable supersolvable oriented matroids}
\label{sec:ApplicationExamples}

We provide examples of non-realizable supersolvable oriented matroids.
In particular, we observe that every oriented matroid of rank $3$ can be extended to
a supersolvable oriented matroid.

Recall the definition of an extension of an oriented matroid (see Definition~\ref{def:OMExtension}).
Before we construct supersolvable extensions, 
we record a fundamental result for arrangements of pseudolines.
This is Levi's Enlargement Lemma \cite{Levi1926_PseudoGeraden} which 
we state in its version for oriented matroids of rank $3$,
see \cite[Prop.~6.3.4]{BLSWZ1999_OrientedMatroids}.
\begin{lemma}[Levi's Enlargement Lemma]
	\label{lem:LevisEnlargement}
	Let $\OM = (E,\Lc)$ be an oriented matroid of rank $3$ and let $X_1,X_2 \in L(\OM)$
	be two distinct flats of rank $2$ with $X_1\cap X_2 = \emptyset$.
	Then there is a one element extension $\wt{\OM}$ on $E \sqcup \{g\}$ of $\OM$
	such that $\wt{X_1},\wt{X_2}$ are flats of rank $2$ and $g \in \wt{X_1} \cap \wt{X_2}$.
	Moreover, the extension can be chosen such that $g \notin \wt{Y}$ for any flat $Y \in L\setminus \{X_1,X_2\}$ of rank $2$.
\end{lemma}	

From the preceding lemma we can delineate the existence of supersolvable extensions
for oriented matroids of rank $3$.
\begin{proposition}
	\label{prop:OMSSExtension}
	Let $\OM$ be an oriented matroid of rank $3$.
	Then there is an extension $\wt{\OM}$ of $\OM$ which is supersolvable.
\end{proposition}
\begin{proof}
	Let $X \in L(\OM)$ be some flat of rank $2$.
	If there is no $Y \in L(\OM)$ of rank $2$ with $X\cap Y = \emptyset$,
	then by Lemma~\ref{lem:SSLrk3} the oriented matroid $\OM$ is already supersolvable and there is nothing to show.
	Otherwise, let $Y \in L(\OM)$ be a flat of rank $2$ with $X \cap Y = \emptyset$.
	By Lemma~\ref{lem:LevisEnlargement} we can extend $\OM$ by a single element $g$ to $\wt{\OM}$
	such that $\wt{X} \cap \wt{Y} = \{g\}$. 
	Then we have 
	\[
		|\{Y' \in L(\wt{\OM}) \mid Y' \cap \wt{X} = \emptyset\}| < |\{Y \in L(\OM) \mid X \cap Y = \emptyset\}|.		
	\]
	Hence, after a finite number of single element extensions via Lemma~\ref{lem:LevisEnlargement} we arrive at
	an extension $\wt{\OM}'$ of $\OM$ which is supersolvable thanks to Lemma~\ref{lem:SSLrk3}.
\end{proof}

\begin{remark}
	\label{rem:SSExtHigherRk}
	Note that there is no analogue of Lemma \ref{lem:LevisEnlargement} for oriented matroids of rank $\geq 4$.
	E.g.\ the Edmonds-Fukuda-Mandel oriented matroid $\mathtt{EFM}(8)$ -- 
	an orientation of the uniform matroid of rank $4$ on $8$ points --
	gives an example that has three corank $1$ flats which cannot be connected
	by any extension, cf.\ \cite[Prop.~10.4.5]{BLSWZ1999_OrientedMatroids}.
	Consequently, it is not clear to us that supersolvable extensions for higher rank oriented matroids
	exist in general.
	In contrast, in the realizable case one easily sees that any hyperplane arrangement is a subarrangement
	of a supersolvable arrangement.
\end{remark}

The remark leaves the following question.

\begin{question}
	\label{que:SSExtHigherRk}
	Do oriented matroids of rank $\geq 4$ always have supersolvable extensions?
\end{question}

Starting with a non-realizable oriented matroid, Proposition~\ref{prop:OMSSExtension} 
yields a non-realizable supersolvable oriented matroid.
Consequently, we obtain many non-realizable supersolvable oriented matroids.
\begin{corollary}
	\label{coro:InfSSNonRealOMs}
	There are infinitely many supersolvable oriented matroids
	which are not realizable.
\end{corollary}

All of these non-realizable oriented matroids have aspherical Salvetti complexes by Theorem~\ref{thm:SSOMKpi1}.

\begin{figure}
	
	\begin{tikzpicture}[scale=1]
		\draw (3.872983346207417,-1.) -- (-3.872983346207417,-1.);  
		\draw (3.872983346207417,1.) -- (-3.872983346207417,1.);  
		\draw (-3.2838821814150108,2.2838821814150108) -- (2.2838821814150108,-3.2838821814150108);  
		\draw (-3.5777087639996634,1.7888543819998317) -- (3.5777087639996634,-1.7888543819998317);  
		\draw (2.2838821814150108,3.2838821814150108) -- (-3.2838821814150108,-2.2838821814150108);  
		\draw (-2.2838821814150108,3.2838821814150108) -- (3.2838821814150108,-2.2838821814150108);  
		\draw (3.5777087639996634,1.7888543819998317) -- (-3.5777087639996634,-1.7888543819998317);  
		\draw (3.2838821814150108,2.2838821814150108) -- (-2.2838821814150108,-3.2838821814150108);  
		\draw (4.,0.) -- (0.5,0);
		\draw (0.5,0) .. controls (0,0.32) .. (-0.5,0);
		\draw (-0.5,0) -- (-4.,0.);  
		
		\draw[dashed] (0.,4.) -- (0.,-4.);  
		\draw[dashed] (1.3776388834631177,3.7552777669262354) -- (-2.1776388834631173,-3.355277766926235);  
		\draw[dashed] (-1.3776388834631177,3.7552777669262354) -- (2.1776388834631173,-3.355277766926235);  
		
		\draw[dashed] (0.96095202129184909,3.8828560638755474) -- (-1.5609520212918493,-3.6828560638755472);  
		
		\draw[dashed] (-0.96095202129184909,3.8828560638755474) -- (1.5609520212918493,-3.6828560638755472);  

		\draw (0.0,-1.) circle[radius=2pt];  
		\fill (0.0,1.) circle[radius=2pt];  
		\draw (-0.66666666666666663,-0.33333333333333331) circle[radius=2pt];  
		\draw (0.0,0.0) circle[radius=2pt];  
		\draw (0.66666666666666663,-0.33333333333333331) circle[radius=2pt];  
		\draw (-0.28,0.14) circle[radius=2pt];
		\draw (0.28,0.14) circle[radius=2pt];
		
	\end{tikzpicture}

	\caption{A projective picture of an arrangement of pseudolines representing a (reorienteation class of a) non-realizable oriented matroid 
		and a supersolvable extension.}
	\label{fig:NonPappusSSOM}
\end{figure}
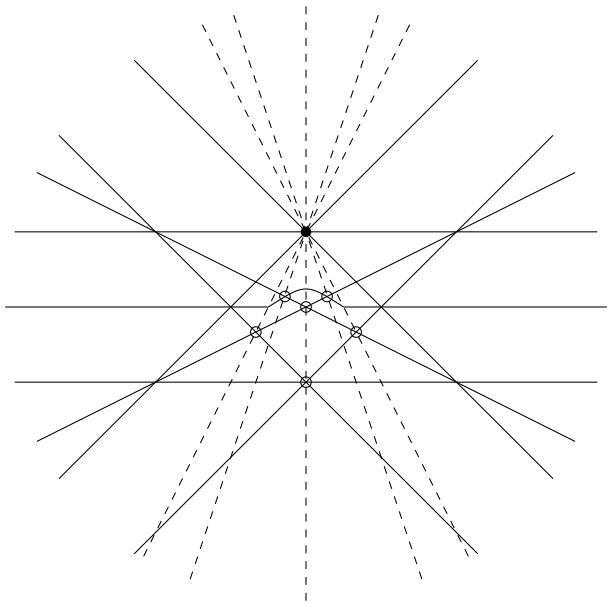

We conclude with a simple example of a non-realizable supersolvable oriented matroid
of rank $3$. For a precise definition of an \emph{arrangement of pseudolines} 
-- a nice visual way to represent (reorientation classes of) rank $3$ oriented matroids, we refer to \cite[Ch.~6]{BLSWZ1999_OrientedMatroids}.
\begin{example}
	\label{ex:NonRealSSOM}
	Let $\OM$ be the oriented matroid of rank $3$
	associated to some orientation of the pseudoline arrangement,
	illustrated as a projective picture in Figure~\ref{fig:NonPappusSSOM}.
	By Pappus' theorem, the oriented matroid is not realizable by an arrangement of linear hyperplanes.
	
	By Proposition~\ref{prop:OMSSExtension}, $\OM$ can be extended to a supersolvable
	oriented matroid $\wt{\OM}$. A possible extension is illustrated by the dotted lines in Figure~\ref{fig:NonPappusSSOM}.
	The intersection point marked by a solid circle represents a modular flat $X$ of corank one. It is connected to
	every other intersection point of pseudolines in the arrangement by some line passing through $X$.
	The resulting supersolvable oriented matroid $\wt{\OM}$ is evidently not realizable.
\end{example}


\section*{Acknowledgments}
The author would like to thank Emanuele Delucchi and Masahiko Yoshinaga for helpful discussions.
He would also like to thank Gerhard R\"ohrle for useful remarks on an earlier draft of the manuscript.

The author is grateful to the JSPS for financial support during a JSPS
fellowship for foreign researchers in Japan.


\newcommand{\etalchar}[1]{$^{#1}$}
\providecommand{\bysame}{\leavevmode\hbox to3em{\hrulefill}\thinspace}
\providecommand{\MR}{\relax\ifhmode\unskip\space\fi MR }
\providecommand{\MRhref}[2]{%
	\href{http://www.ams.org/mathscinet-getitem?mr=#1}{#2}
}
\providecommand{\href}[2]{#2}


\end{document}